\documentclass[11pt,preprint,noinfoline]{imsart}
\usepackage{amsmath,amssymb,amsthm,booktabs}
\usepackage{mathrsfs}

\usepackage[utf8]{inputenc}
\usepackage[colorlinks,citecolor=blue,urlcolor=blue]{hyperref}

\usepackage{bm}\usepackage{euscript}
\usepackage{graphicx}

\usepackage{multicol}
\usepackage[usenames,dvipsnames,svgnames,table]{xcolor}

\usepackage[authoryear]{natbib}
\usepackage{a4wide}

\usepackage{mathrsfs}
\usepackage{euscript}

\numberwithin{equation}{section}

\allowdisplaybreaks[4]

\theoremstyle{plain}
\newtheorem{theorem}{Theorem}[section]
\newtheorem{lemma}{Lemma}[section]

\newtheorem{definition}{Definition}[section]
\newtheorem{remark}{Remark}[section]

\def\bC{\mathbb C}
\def\bR{\mathbb R}
\def\bN{\mathbb N}
\def\bP{\mathbb P}
\def\bE{\mathbb E}
\def\bS{\mathbb S}

\def\cA{\mathcal A}
\def\cP{\mathcal P}
\def\cS{\mathcal S}
\def\cE{\mathcal E}
\def\cL{\mathcal L}
\def\bfa{\textbf{a}}
\def\bfx{\textbf{x}}

\def\Comp{\mathrm{Comp}}
\def\Incomp{\mathrm{Incomp}}
\def\dist{\mathrm{dist}}
\def\Num{\mathrm{Num}}
\def\Den{\mathrm{Den}}
\def\Span{\mathrm{Span}}

\def\Tr{\mathrm{Tr}}
\def\Var{\mathrm{Var}}
\def\Cov{\mathrm{Cov}}

\newcommand\bx{\mathbf{x}}
\newcommand\Xn{X^{(n)}}
\newcommand\Yn{Y^{(n)}}
\newcommand\Zn{Z^{(n)}}

\begin{document}

\begin{frontmatter}

\title{On eigenvalue distributions of large auto-covariance matrices}
  
\runtitle{Eigenvalue distributions of large auto-covariance matrices}

\begin{aug}
	\author{\fnms{~ Jianfeng} \snm{Yao~}\ead[label=e2]{jeffyao@hku.hk}}
	\and
	\author{\fnms{~ Wangjun} \snm{Yuan}\ead[label=e3]{ywangjun@connect.hku.hk}}
	
    \affiliation{The University of Hong Kong}
	\runauthor{J. Yao \&  W. Yuan}

	\address{ % Jianfeng Yao\\
	Department of Statistics and Actuarial Science\\
	The University of Hong Kong\\
	\printead{e2}
	}

	\address{ % Wangjun Yuan
	Department of Mathematics \\
	The University of Hong Kong\\
	\printead{e3}
	}   

\end{aug}

\begin{abstract}
	In this article, we establish a limiting distribution for eigenvalues of a class of auto-covariance matrices. The same distribution has been found in the literature for a regularized version of these auto-covariance matrices. The original non-regularized auto-covariance matrices are non invertible which introduces supplementary difficulties for the study of their eigenvalues through Girko's Hermitization scheme. The key result in this paper is a new polynomial lower bound for the least singular value of the resolvent matrices associated to a rank-defective quadratic function of a random matrix with independent and identically distributed entries. Another improvement in the paper is that the lag of the auto-covariance matrices can grow to infinity with the matrix dimension.
\end{abstract}
  


\end{frontmatter}

\section{Introduction}

Most of matrix ensembles studied in random matrix theory are either Hermitian or unitary. General random matrices without such symmetry or invariance are much less studied as their eigenvalues can lay everywhere in the complex plane. Typically eigenvalues of non-Hermitian matrices are much more unstable than the Hermitian ones, and need new mathematical tools for their study. These tools have taken long time to emerge as illustrated by the history of the circular law. The law states that the empirical spectral distribution of a $n\times n$ random matrix  with i.i.d entries of mean 0 and variance $1/n$ converges to the uniform distribution on the unit disk of the plane. When the common distribution is complex Gaussian and as early as in 1964, \cite{Ginibre1964} established the circular law in expectation (although the abstract of the paper ended with a quite confusing statement that ``the limit of the eigenvalue density as $n\to\infty$ is constant over the whole complex plane''). The real Gaussian equivalent of Ginibre's result was later established by \cite{Edelman1997}. The general non-Gaussian case was first tackled by \cite{Girko1984} who coined  the name of circular law and more importantly, introduced a powerful analytic tool known as Girko's Hermitization which is followed in all the subsequent papers. His own result remains however controversial as few key steps in his argument, were deemed not fully justified at the time \cite[Section 2.8]{Tao2012}. \cite{Bai1997} provided a rigorous proof of the circular law assuming a few additional conditions on the moments and density functions of the entries. These non necessary conditions are afterwards removed in several subsequent papers before reaching the final circular law with minimal conditions established in \cite{Tao2010}. We refer to \cite{Bordenave2012} for a more detailed description on these different episodes of the circular law.

One immediately notices that such multi-decade long effort was only about the simplest non-Hermitian matrix filled with i.i.d. entries. Recent progress is made with a more involved matrix ensemble known as {\em structured random matrices} of the form $A_n\circ X^{(n)} + B_n $ where the random matrix $ X^{(n)}$ has i.i.d. entries as in the circular law and $(A_n)$ and $(B_n)$ are two sequences of deterministic matrices (the $\circ$ denotes the Hadamard product), see \cite{Cook2018}.   This model includes the class of random matrices with profile \cite{Cook2018b} with $B_n\equiv 0$ and $A_n(i,j)=\sigma(i/n,j/n)$ for $(i,j)\in [n]^2$ and a scalar profile function $\sigma: (0,1)^2 \to (0,\infty)$, and the class of band matrices with $B_n\equiv 0$ and $A_n(i,j)=0$ if $|i-j|>k $ for some given bandwidth $k\ge 1$. The limiting distribution of random matrices with profiles is obtained in \cite{Cook2018b} where a key ingredient is a polynomial lower bounds for the least singular value of the associated resolvent matrices established in \cite{Cook2018}.

In this paper we study a particular random matrix of the following form. Let $N = N(n)$ and $ 1\le k = k(n) < n$ be sequences of positive integers varying with $n \in \bN_+$. For each $n$, consider a rectangular $N\times n$ random matrix $X^{(n)} =(X_{ij}^{(n)}  )_{1 \le i \le N, 1 \le j \le n} \in \bC^{N \times n}$ with i.i.d. complex-valued entries with mean 0 and variance $1/n$. The matrix of interest is
\begin{align} \label{def-Y}
  Y^{(n)} = 
  X^{(n)}  A^{(n)} (X^{(n)} )^*, 
\end{align}
where $(A^{(n)})$ is the sequence of deterministic matrices with entries $A_{ij}^{(n)} = \delta_{i=j+k}$, i.e. $A^{(n)}$ is of the form
\begin{align} \label{def-matrix A}
  A^{(n)} = \left(
  \begin{matrix}
	0 & 0 \\
	I_{n-k} & 0
  \end{matrix}
  \right),
\end{align}

The non-Hermitian matrix $Y^{(n)}$ originates from high-dimensional time series analysis. Write   $X^{(n)}$  in function of its $n$ column vectors of dimension $N$ as  in  $ \Xn=(\bx_1,\ldots\bx_n)$. Clearly,
\begin{align*}
	\Yn = \bx_{k+1} \bx_1^* +  \cdots + \bx_n\bx_{n-k}^*.
\end{align*}
In this form the matrix $\Yn$ is  seen as  the so-called {\em lag-$k$ auto-covariance matrix} of the time series $(\bx_j)_{1\le j \le n}$ in the space $\bC^N$ and observed at time $1\le j \le n$. The case of $k=1$ would be of special interest and we denote it as $\Yn_{1}$. The spectral property of the matrix has a fundamental role for the analysis of the series. For example it helps to test the hypothesis whether the series is a {\em white noise}, that is the series is indeed an i.i.d. sequence. The limiting distribution of the singular values of $Y_1^{(n)}$ has been found in \cite{Li2015,Wang2016}; it has been applied to high-dimensional statistics in \cite{Li2017,Li2019}.

To fix the discussions, throughout  the paper the dimension parameters are taken to satisfy the following asymptotic scheme:
\begin{align} \label{eq-1.0-N/n}
	\lim_{n \rightarrow \infty} \dfrac{N}{n} = \gamma_0 \in (0,\infty), \quad
	\lim_{n \rightarrow \infty} \dfrac{k}{n} = \gamma_1 \in [0,1).
\end{align}
Very few is known on the eigenvalue distribution of the matrix $\Yn$. Simulation and plots are given in the book \cite{BoseBook} for (see Figure 8.1 there).
In a recent paper \cite{Bose2019}, the authors consider a variant of $\Yn_{1}$, namely
\begin{equation} \label{def-Z}
	Z^{(n)} =\bx_{2} \bx_1^* + \cdots + \bx_n \bx_{n-1}^* + \bx_1 \bx_{n}^*
	= \Yn_{1} +\bx_1 \bx_{n}^*,
\end{equation}
and established that the empirical spectral distribution of $\Zn$ converges to  a deterministic limiting distribution $\mu^{(\gamma_0)}$ in probability (this distribution will be detailed later).
Therefore the two matrices $\Yn_{1}$ and $\Zn$ differ only by the rank-one matrix $\bx_1 \bx_{n}^*$. However due to the already mentioned high spectral instability, rank-one perturbations can preserve or destroy the spectrum of the original matrix depending on their nature.
In other words, the existence of the LSD $\mu^{(\gamma_0)}$ for $Z^{(n)}$ does not imply anything a priori on the asymptotic properties of the lag-1 auto-covariance matrix $\Yn_{1}$.

Actually the present paper aims at establishing the LSD for the general lag-$k$ auto-covariance matrix $\Yn$ under reasonable conditions.
Note that while in \cite{Bose2019}, the matrix $Z^{(n)}$ corresponds
to the case $k=1$,  we allow $k$ growing with $n$ in this paper.
Technically, by mimicking the methodology introduced in the development of the circular law, the main technical challenge here  is to establish a polynomial lower bounds for the least singular value of the resolvent matrix $ \Yn-zI_N$ for almost all $z\in\bC$. 
Consider for a moment the method employed in \cite{Bose2019}  for the establishment of the LSD $\mu^{(\gamma_0)}$ for the matrix $\Zn$.
Note that this matrix can be rewritten as
\begin{align*}
	\Zn = X^{(n)}  J ^{(n)}(X^{(n)} )^*,
\end{align*}
with the permutation matrix
\begin{align*}
	J^{(n)} = \left(
	\begin{matrix}
		0 & 1 \\
		I_{n-1} & 0
	\end{matrix}
	\right).
\end{align*}
In a setting where the entries of $X^{(n)}$ have a (common) density and noting that $J^{(n)}$ is of full rank, the matrix $\Zn$ is of full rank almost surely.
This is a main ingredient for the method in \cite{Bose2019} to establish a polynomial lower bound for the least singular value of the corresponding resolvent $\Zn - zI_N$ ($z\in\bC$).
(Note that in the reference such  polynomial lower bound is established for  general degree 2 monomials of the form $X^{(n)} C^{(n)} (X^{(n)})^*$ where  $C^{(n)}$ is asymptotically non degenerated).
This method is broken in our case of $\Yn$  since the inner matrix  $A^{(n)}$ in $\Yn$ is nilpotent and of rank $n-k$. We thus introduce a specially designed non-degenerated approximation $H(z) \in \bC^{(N+n-k) \times (N+n-k)}$ to the resolvent with a smaller least singular value than the resolvent  $\Yn - zI_N$. A careful analysis leads to a manageable polynomial bound for the least singular value of  $H(z)$ which is thus easily transferred to the resolvent $\Yn - zI_N$. This construction of a lower bound for  the resolvent is indeed the main technical innovation of the paper. It is developed in Section~\ref{sec-least singular value}. Note that in the case of the circular law or the structured matrices of \cite{Cook2018}, the matrix is linear in its independent entries. In contrary, the matrix $Y^{(n)}$ as well as the matrix $Z^{(n)}$ is a more involved quadratic function of these independent entries.
Note that complex Gaussian valued auto-covariance matrices were also considered in \cite{Wojciech2017}.

The rest of the paper is as follows. Section~\ref{sec-preliminary} recalls a few  preliminaries and useful results from the literature.
Section~\ref{sec-least singular value} presents the main result of the paper, that is, a polynomial bound on the least singular value of the resolvent $\Yn - zI_N$. 
Applying this bound leads to the LSD for the matrix $\Yn$ in Section~\ref{sec-LSD}.
The two appendices  collect a few useful but standard lemmas from linear algebra and probability theory.

Below are some useful notations.

\begin{itemize}
\item
  A ball with center $z \in \bC$
  and radius  $r \ge 0$  is denoted as
  $B(z,r)$.
  Let $T$ be a set of complex numbers, then $B(T,r) = \cup_{z\in T} B(z,r)$.
\item 
  For an integer $n$, set $[n] = [1,n] \cap \bN_+$.
  For a vector $u \in \bC^n$ and a set $I \subseteq [n]$ of integers,  $u_I$ is  the sub-vector of $u$ with indexes in $I$.
  Similarly for  a matrix $M \in \bC^{n \times n}$ and  index sets $I,
  J \subseteq [n]$,  $M_{I,J}$  denotes  the submatrix of $M$
  restricted to 
  rows with  index  in $I$ and columns with index in $J$.
  In the case that the set $I$  contains one element $i_0$ only, we
  may write $M_{i_0,J}$. The same abbreviation also  applies to the
  column index set $J$.
\item 
  For any index set $I \subset [n]$, let $\Pi_I: \bC^n \rightarrow
  \bC^n$ be a projection such that $(\Pi_Iu)_i = u_i 1_{i \in I}$.
\item 
  Without further indication, all vectors in this paper are column vectors. We denote $\left\{ e_1^{(n)}, \ldots, e_n^{(n)} \right\}$ as the standard base of $\bC^n$, i.e. $e_i^{(n)}$ is the $i$-th column of $I_n$.
\item 
  For a given matrix $M \in \bC^{p \times q}$, let
  $s_1(M) \ge  \cdots\ge  s_r(M)$ be the ordered singular values,
  where $r = \min\{p,q\}$. We use the convention that $s_j(M) = 0$ for
  $j > r$. We also denote by $\nu_M = \frac{1}{p} \sum_{j=1}^r
  \delta_{s_j(M)}$  the singular values empirical distribution. If $p
  = q$, then  $\{\lambda_1(M), \ldots, \lambda_p(M)\}$ denote  the set
  of eigenvalues of $M$ and $\mu_M = \frac{1}{p} \sum_{j=1}^p
  \delta_{\lambda_j(M)}$  the  corresponding  empirical spectral  distribution.
\item
	For a given matrix $M$, denote $\|M\|$ as the operator norm and $\|M\|_{HS}$ as the Hilbert-Schmidt norm of $M$. 
\item
	Denote $\iota = \sqrt{-1}$.
\end{itemize}

\section{Preliminaries} \label{sec-preliminary}

To ease the reading of the proofs in Section~\ref{sec-least singular value} and  \ref{sec-LSD},
we collect the main existing concepts and results from the literature that will be
used afterwards. 

\subsection{Compressible vectors and incompressible vectors}

For $\theta, \rho \in (0,1)$, we define the set of compressible vectors
\begin{align*}
	\Comp(\theta,\rho) &=\bS^{n-1} \cap \bigcup_{I \subseteq [n], |I| = \theta n} B(\bS^{n-1}_I,\rho) \\
	&= \{u \in \bS^{n-1}: \exists \ J \in [n], |J| = \theta n, \exists \ v \in \bS^{n-1}, \ s.t. \ v_{[n] \setminus J} = 0 \ \mathrm{and} \ \|u-v\| \le \rho \},
\end{align*}
and the set of incompressible vectors
\begin{align*}
	\Incomp(\theta,\rho) = \bS^{n-1} \setminus \Comp(\theta,\rho).
\end{align*}

The following lemma is the structure of the set of Incompressible vectors, which could be found in \cite[Lemma 3.4]{Mark2008} or \cite[Lemma 2.1]{Cook2018}.

\begin{lemma}{\cite[Lemma 3.4]{Mark2008}} \label{Lemma-Incompressible-weak}
For $\theta, \rho \in (0,1)$, for any $u = (u_1, \ldots, u_n)^\intercal \in \Incomp(\theta,\rho)$, the set
\begin{align*}
	J = \left\{ i \in [n]: \dfrac{\rho}{\sqrt{n}} \le |u_i| \le \dfrac{2}{\sqrt{\theta n}} \right\}
\end{align*}
has cardinal number  $|J| \ge 3 \theta n/4$.
\end{lemma}

Lemma \ref{Lemma-Incompressible-weak} can be slightly extend to the following lemma, which could be found in \cite{Bose2019}.

\begin{lemma}{\cite[Lemma 8]{Bose2019}} \label{Lemma-Incompressible}
For $\theta, \rho \in (0,1)$, for any $u = (u_1, \ldots, u_n)^\intercal \in \Incomp(\theta,\rho)$, and $\tilde{u} = (\tilde{u}_1, \ldots, \tilde{u}_n)^\intercal \in \bS^{n-1}$, the set
\begin{align*}
	J' = \left\{ i \in [n]: \dfrac{\rho}{\sqrt{n}} \le |u_i| \le \dfrac{2}{\sqrt{\theta n}}, |\tilde{u}_i| \le \dfrac{2}{\sqrt{\theta n}} \right\},
\end{align*}
has cardinal number  $|J'| \ge \theta n/2$.
\end{lemma}

\begin{proof}
Denote
\begin{align*}
	J'' = \left\{ i \in [n]: |\tilde{u}_i| \le \dfrac{2}{\sqrt{\theta n}} \right\},
\end{align*}
then $|(J'')^\complement| \le \theta n/4$. As $J' = J \cap J''$, $|J'| \ge |J| - |(J'')^\complement| \ge \theta n/2$, by Lemma \ref{Lemma-Incompressible-weak}.
\end{proof}

The following lemma is the invertibility via distance, which could be found in \cite{Mark2008}.

\begin{lemma}{\cite[Lemma 3.5]{Mark2008}} \label{Lemma-Incompressible distance}
Let $A \in \bC^{n \times n}$ be any random matrix. Let $A_{-k}$ be the span of all column vectors of $A$ except the $k$-th column. Then for every $\theta, \rho \in (0, 1)$ and every $\epsilon > 0$, one has
\begin{align*}
	\bP \left( \inf_{x \in \Incomp(\theta, \rho)} \|Ax\| < \dfrac{\rho t}{\sqrt{n}} \right)
	\le \dfrac{1}{\theta n} \sum_{k=1}^n \bP \left( \dist(A_{[n],k}, A_{-k}) \le t \right).
\end{align*}
\end{lemma}

The following lemma deals with the distance of the columns of a matrix, which could be found in \cite{Bose2019}.

\begin{lemma}{\cite[page 5]{Bose2019}} \label{Lemma-distant}
Let $A \in \bC^{n \times n}$. For $k \in [n]$, $A_{-k}$ be the space spanned by $A_{[n],i}$ with $i \neq k$. Then for $k \in [n]$
\begin{align*}
	\dist(A_{[n],k}, A_{-k})
	= \dfrac{\left| A_{kk} - A_{k,[n] \setminus \{k\}} \left( A_{[n] \setminus \{k\}, [n] \setminus \{k\}} \right)^{-1} A_{[n] \setminus \{k\}, k} \right|}{\sqrt{1 + \left\| A_{k,[n] \setminus \{k\}} \left( A_{[n] \setminus \{k\}, [n] \setminus \{k\}} \right)^{-1} \right\|^2}}.
\end{align*}
\end{lemma}

The following lemma is the metric entropy of the sphere, which is introduced in \cite{Cook2018}.

\begin{lemma}{\cite[Lemma 2.2]{Cook2018}}\label{Lemma-Compressible}
Let $V \subseteq \bC^n$ be a subspace of (complex) dimension $k$, and let $T \subseteq V \cap \bS^{n-1}$. For $\rho \in (0,1)$, $T$ has a $\rho$-net $\Sigma \subseteq T$ of cardinality $|\Sigma| \le (3/\rho)^{2k}$.
\end{lemma}

\subsection{Small ball probability}

The following definition is the small ball probability, which could be found in \cite{Tao2008}.
\begin{definition}{\cite[Definition 3.1]{Tao2008}} \label{Def-small ball prob}
Let $Z = (Z_1, \ldots, Z_n)$ be a complex random vector with independent entries. For any $r \ge 0$, we define the small ball probability
\begin{align*}
	\cS \left( \sum_{j=1}^n Z_j, r \right)
	= \sup_{z \in \bC} \bP \left( \sum_{j=1}^n Z_j \in B(z,r) \right).
\end{align*}
\end{definition}

The following lemma states that the small ball probability is monotone with respect to the dimension.

\begin{lemma} {\cite[Lemma 2.1]{Mark2008}} \label{Lemma-small ball prob monotone}
For any $r \ge 0$ and any index set $I \subseteq [n]$,
\begin{align*}
	\cS \left( \sum_{j \in [n]} Z_j, r \right)
	\le \cS \left( \sum_{j \in I} Z_j, r \right).
\end{align*}
\end{lemma}

The following lemma is the Berry-Esseen theorem for small ball probability, which could be found in \cite{Bordenave2012}.

\begin{lemma} {\cite[Lemma A.6]{Bordenave2012}} \label{Lemma-small ball prob-upper bound}
Suppose that the independent complex random variables $Z_1, \ldots, Z_n$ are centered with finite third moments. Then there exists a constant $c' > 0$, such that
\begin{align*}
	\cS \left( \sum_{j=1}^n Z_j, r \right)
	\le \dfrac{c'r}{\sqrt{\sum_{j=1}^n \bE [|Z_j|^2]}} + \dfrac{c'\sum_{j=1}^n \bE [|Z_j|^3]}{\left( \sum_{j=1}^n \bE [|Z_j|^2] \right)^{3/2}}.
\end{align*}
\end{lemma}

The following lemma estimate the probability that quadratic form is bounded, which is introduced in \cite{Bose2019}.

\begin{lemma} {\cite[Lemma 20]{Bose2019}} \label{Lemma-quadratic form bound}
Let $a \in \bC$, $u,v \in \bC^n$, $M \in \bC^{n \times n}$ be deterministic. Let $Z \in \bC^n$ be a random variable with independent entries, and $Z'$ be an independent copy of $Z$. Let $I \subseteq [n]$, then for each $t > 0$,
\begin{align*}
	\bP \left( \left| Z^* M Z + u^*Z + Z^* v + a \right| \le t \right)^2
	\le \bE_{Z_{I^\complement},Z_{I^\complement}'} \left[ \cS_{Z_I} \left( \left( Z_{I^\complement} - Z_{I^\complement}' \right)^* M_{I^\complement,I} Z_I + Z_I^* M_{I,I^\complement} \left( Z_{I^\complement} - Z_{I^\complement}' \right), 2t \right) \right].
\end{align*}
Here, $\cS_{Z_I}$ is the small ball probability defined in Definition \ref{Def-small ball prob}, where the expectation is taken with respect to $Z_I$. We also use the convention that the right hand side is $1$ if $I = \emptyset$ or $[n]$.
\end{lemma}

\subsection{Logarithmic potential}
Let $\cP(\bC)$ be the set of probability measures on $\bC$ which $\ln|x|$ is integrable in a neighbourhood of infinity. 

\begin{definition} \label{Def-log potential}
The logarithmic potential $U_{\mu}$ of $\mu \in \cP(\bC)$ is the function defined by
\begin{align*}
	\cL_{\mu}(z) = - \int_{\bC} \ln |z-\lambda| d\mu(\lambda), z \in \bC.
\end{align*}
\end{definition}

The following lemmas are from \cite{Bordenave2012}
\begin{lemma} {\cite[Lemma 4.1]{Bordenave2012}} \label{Lemma-unique of log potential}
For $\mu,\nu \in \cP(\bC)$, if $\cL_{\mu} = \cL_{\nu}$ almost everywhere, then $\mu = \nu$.
\end{lemma}

\begin{lemma} {\cite[Lemma 4.3]{Bordenave2012}} \label{Lemma-Hermitization}
Let $B_n \in \bC^{n \times n}$ be a complex random matrix. Suppose that there exists a family of non-random probability measures $\{\nu_z: z \in \bC\}$ on $\bR_+$, such that $\nu_{B_n-zI_n}$ converges to $\nu_z$ in probability as $n \rightarrow \infty$, and the function $\ln(x)$ is uniformly integrable for the family $\{\nu_{B_n-zI_n}: n \in \bN_+\}$ in probability, for almost all $z \in \bC$. Then there exists a probability measure $\mu \in \cP(\bC)$, such that $\mu_{B_n}$ converges to $\mu$ in probability, and
\begin{align*}
	\cL_{\mu}(z) = - \int_0^{\infty} \ln(\lambda) d\nu_z(\lambda).
\end{align*}
\end{lemma}

\section{Least singular value estimation} \label{sec-least singular value}

In this section, denote $C$ be a large positive constant and $c$ be a small positive constant that may vary in different place and may depend on $z$. The dimension $n$ is very large and is fixed. We may omit it in the superscript without ambiguity. Moreover, by \eqref{eq-1.0-N/n}, we may assume that $N/n \in (\gamma_0/2, 3\gamma_0/2)$ and $k < c_0'n$ for a positive constant $c_0'<1$.

We impose the following conditions on the random variable $X_{11}^{(n)}$.
\begin{enumerate}
	\item [(C1)] For $n \in \bN_+$, the complex random variables $\{X_{ij}^{(n)}: i \in [N], j \in [n]\}$ are i.i.d. with mean zero and variance $1/n$. Besides, there exists a positive constant $m_4$, such that $n^2 \bE \left[ \left| X_{11}^{(n)} \right|^4 \right] \le m_4$ for all $n \in \bN$.
	\item[(C2)] $c_0 = 1 - \sup_{n \in \bN_+} \left| n \bE \left[ \left(X_{11}^{(n)} \right)^2 \right] \right| > 0$.
\end{enumerate}
Assumption (C1) is the standard moment conditions on the matrix entries.
Assumption (C2) means that the complex random variable $X_{11}^{(n)}$ cannot be supported on a line passing by the origin. This assumption facilitates the least singular value estimation using projection arguments.

\begin{theorem} \label{Thm-least singular value}
Assume that the conditions (C1) and (C2) hold. Then there exists a positive constant $C$ that depends on $z$ and $C_0$, such that for all $n$ large,
\begin{align} \label{eq-thm-least singular value}
	\bP \left( s_N(XAX^* - zI_N) \le n^{-37/22}, \|X\| \le C_0 \right) \le C n^{-1/22},
\end{align}
for $z \in \bC \setminus \{0\}$.
\end{theorem}

\noindent\underline{Outline of the proof of Theorem \ref{Thm-least singular value}}: \quad
We use the inverse formula for blocking matrices to find an invertible matrix $H \in \bC^{(N+n-k) \times (N+n-k)}$, such that $s_{N+n-k}(H) \le s_N(XAX^* - zI_N)$. To estimate the least singular value of $H$, we notice that $s_{N+n-k}(H) = \inf_{w \in \bS^{N+n-k-1}} \|Hw\|$. Then we may estimate the infimum over compressible vectors by using an $\epsilon$-net argument in Section~\ref{step1}, and incompressible vectors by using Lemma \ref{Lemma-Incompressible distance} in Section~\ref{step2} (the case $2k+1 \le n$) and Section \ref{step3} (the case $2k+1 > n$), respectively. In Section \ref{step2}, due to the structure of $H$, we estimate the distance $\dist(H_{[N+n-k],l}, H_{-l})$ for $N+k+1 \le l \le N+n-k$ in Section~\ref{Case1}, for $N+1 \le l \le N+k$ in Section~\ref{Case2} and for $1 \le l \le N$ in Section~\ref{Case3}, respectively. In Section \ref{step3}, we estimate the distance $\dist(H_{[N+n-k],l}, H_{-l})$ for $N+1 \le l \le N+n-k$ in Section~\ref{Case4} and for $1 \le l \le N$ in Section~\ref{Case5}, respectively.

\bigskip 

\noindent To start the proof, we can assume that $X_{11}$ has density by a perturbation argument (see \cite{Bose2019}).
We fix arbitrary $z \in \bC \setminus \{0\}$. All constants in the proof may depend on $z$.
Denote $X = (X_1, \ldots, X_n)$ then we have
\begin{align*}
	XAX^* = (X_1, \ldots, X_n) A \left(
	\begin{matrix}
	X_1^* \\
	\vdots \\
	X_n^*
	\end{matrix}
	\right)
	= (X_{k+1}, \ldots, X_n, 0, \ldots, 0) \left(
	\begin{matrix}
	X_1^* \\
	\vdots \\
	X_n^*
	\end{matrix}
	\right)
	= (X_{k+1}, \ldots, X_n) \left(
	\begin{matrix}
	X_1^* \\
	\vdots \\
	X_{n-k}^*
	\end{matrix}
	\right).
\end{align*}
Denote
\begin{align*}
	H' = \left(
	\begin{matrix}
	zI_N & (X_{k+1}, \ldots, X_n) \\
	(X_1, \ldots, X_{n-k})^* & I_{n-k}
	\end{matrix}
	\right) \in \bC^{(N+n-k) \times (N+n-k)}.
\end{align*}
Set $H = H'$ if $2k+1>n$, and
\begin{align*}
	&H = \left(
	\begin{matrix}
	zI_N & (X_{n-k+1}, \ldots, X_n, X_{k+1}, \ldots, X_{n-k}) \\
	(X_1, \ldots, X_{n-k})^* & \left( e_{n-2k+1}^{(n-k)}, \ldots, e_{n-k}^{(n-k)}, e_1^{(n-k)}, \ldots, e_{n-2k}^{(n-k)} \right)
	\end{matrix}
	\right) \in \bC^{(N+n-k) \times (N+n-k)},
\end{align*}
if $2k+1\le n$.

Next, we show that $zI_N - XAX^*$ is invertible with probability $1$. Indeed, we define a function $f(W) = \det \left( zI_N - WAW^* \right)$ for $W = (W_1, \ldots, W_n) \in \bC^{N \times n}$, then it is a polynomial of the entries of $W$. Moreover, we have $f\left( \left( e_1^{(N)}, \ldots, e_1^{(N)} \right) \right) = (z-n+k) z^{N-1} \not= z^N = f(0)$, which implies that $f(W)$ is a non-zero polynomial of the entries of $W$. Thus, the polynomial hyper-surface $\{W:f(W) = 0\}$ has zero Lebesgue measure in $\bC^{N \times N}$. Since the entries of $X$ have density, $f(X) \not= 0$ almost surely. By lemma \ref{Lemma-Inverse of block matrix}, we have
\begin{align*}
	\left(H'\right)^{-1} = \left(
	\begin{matrix}
	\left( zI_N - XAX^* \right)^{-1} & * \\
	* & *
	\end{matrix}
	\right)
\end{align*}
Thus, by Lemma \ref{Lemma-singular value interacting},
\begin{align} \label{eq-3.2-least singular H'}
	s_N \left( XAX^* - zI_N \right)
	= \dfrac{1}{s_1 \left( \left( XAX^* - zI_N \right)^{-1} \right)}
	\ge \dfrac{1}{s_1 \left( \left( H' \right)^{-1} \right)}
	= s_{N+n-k} \left( H' \right)
\end{align}
Note that one can obtain $H$ from $H'$ by permuting the columns. Thus, the sets of singular values of $H$ and $H'$ are exactly the same. Hence, by \eqref{eq-3.2-least singular H'}, it is enough to show
\begin{align}\label{eq-3.3-least singular value H}
	\bP \left( s_{N+n-k}(H) \le n^{-37/22}, \|X\| \le C_0 \right) \le C n^{-1/22},
\end{align}
for all $z \in \bC \setminus \{0\}$.

Note that for any $\theta, \rho \in (0,1)$,
\begin{align*}
  s_{N+n-k}(H)
  = \inf_{w \in \bS^{N+n-k-1}} \|Hw\|
  = \inf_{w \in \Comp(\theta, \rho)} \|Hw\| \wedge \inf_{w \in \Incomp(\theta, \rho)} \|Hw\|.
\end{align*}
Then the conclusion of the theorem will follow from the following two key estimates:
\begin{align} \label{eq-3.4-Compressible}
	\bP \left( \inf_{w \in \Comp(\theta, \rho)} \|Hw\| \le c, \|X\| \le C_0 \right) \le \exp (-cn),
\end{align}
and
\begin{align} \label{eq-3.5-Incompressible}
	\bP \left( \inf_{w \in \Incomp(\theta, \rho)} \|Hw\| \le n^{-37/22}, \|X\| \le C_0 \right) \le C n^{-1/22},
\end{align}
separably for a pair of special $\theta$ and $\rho$.
These estimates are established in next subsections, respectively.

\subsection{Estimate \texorpdfstring{\eqref{eq-3.4-Compressible}}{(3.4)} for compressible vectors}\label{step1}

We derive the proof of \eqref{eq-3.4-Compressible} for the case $2k+1 \le n$ first.

For a deterministic vector $w = (u^\intercal, v^\intercal)^\intercal \in \bS^{N+n-k-1}$, where $u \in \bC^N$ and $v \in \bC^{n-k}$, for $0 \le t \le c$, by Lemma \ref{Lemma-concentration to space}, we have
\begin{align} \label{eq-3.6-u large}
	& \bP \left( \|Hw\| \le t, \|u\| > 1/2 \right) \nonumber \\
	\le& \bP \left( \left\| (X_1, \ldots, X_{n-k})^* u + \left( e_{n-2k+1}^{(n-k)}, \ldots, e_{n-k}^{(n-k)}, e_1^{(n-k)}, \ldots, e_{n-2k}^{(n-k)} \right) v \right\| \le t, \|u\| > 1/2 \right) \nonumber \\
	\le& \bP \left( \left\| (X_1, \ldots, X_{n-k})^* \dfrac{u}{\|u\|} + \dfrac{1}{\|u\|} \left( v_{k+1}, \ldots, v_{n-k}, v_1, \ldots, v_{k} \right)^\intercal \right\| \le 2t, \|u\| > 1/2 \right) \nonumber \\
	\le& \bP \left( \dist \left( (X_1, \ldots, X_{n-k})^* \dfrac{u}{\|u\|}, \Span \left\{ \left( v_{k+1}, \ldots, v_{n-k}, v_1, \ldots, v_{k} \right)^\intercal \right\} \right) \le 2t \right) \nonumber \\
	\le& \exp (-cn).
\end{align}
Similarly, by Lemma \ref{Lemma-concentration to space}, we have
\begin{align} \label{eq-3.7-v large}
	& \bP \left( \|Hw\| \le t, \|v\| > 1/2 \right) \nonumber \\
	\le& \bP \left( \left\| zu + (X_{n-k+1}, \ldots, X_n, X_{k+1}, \ldots, X_{n-k}) v \right\| \le t, \|v\| > 1/2 \right) \nonumber \\
	\le& \bP \left( \left\| \dfrac{z}{\|v\|} u + (X_{n-k+1}, \ldots, X_n, X_{k+1}, \ldots, X_{n-k}) \dfrac{v}{\|v\|} \right\| \le 2t, \|v\| > 1/2 \right) \nonumber \\
	\le& \bP \left( \dist \left( (X_{n-k+1}, \ldots, X_n, X_{k+1}, \ldots, X_{n-k}) \dfrac{v}{\|v\|}, \Span \left\{ u \right\} \right) \le 2t \right) \nonumber \\
	\le& \exp (-cn).
\end{align}
Thus, by \eqref{eq-3.6-u large} and \eqref{eq-3.7-v large}, for $0 < t < c$, 
\begin{align} \label{eq-3.8-Hw large}
	\bP \left( \|Hw\| \le t \right)
	\le \bP \left( \|Hw\| \le t, \|u\| > 1/2 \right) + \bP \left( \|Hw\| \le t, \|v\| > 1/2 \right)
	\le \exp (-cn).
\end{align}

Note that on the event $\{\|X\| \le C_0\}$,
\begin{align} \label{eq-3.9'-||H||}
	\|H\|
	=& \left\| \left(
	\begin{matrix}
	zI_N & (X_{n-k+1}, \ldots, X_n, X_{k+1}, \ldots, X_{n-k}) \\
	(X_1, \ldots, X_{n-k})^* & \left( e_{n-2k+1}^{(n-k)}, \ldots, e_{n-k}^{(n-k)}, e_1^{(n-k)}, \ldots, e_{n-2k}^{(n-k)} \right)
	\end{matrix}
	\right) \right\| \nonumber \\
	\le& \left\| \left(
	\begin{matrix}
	zI_N & 0 \\
	0 & \left( e_{n-2k+1}^{(n-k)}, \ldots, e_{n-k}^{(n-k)}, e_1^{(n-k)}, \ldots, e_{n-2k}^{(n-k)} \right)
	\end{matrix}
	\right) \right\| \nonumber \\
	& + \left\| \left(
	\begin{matrix}
	0 & (X_{n-k+1}, \ldots, X_n, X_{k+1}, \ldots, X_{n-k}) \\
	(X_1, \ldots, X_{n-k})^* & 0
	\end{matrix}
	\right) \right\| \nonumber \\
	\le& |z| + 1 + C_0.
\end{align}
Thus, we choose $r_H$ to be a small number satisfying
\begin{align*}
	r_H < \dfrac{s_0}{4(|z| + 1 + C_0)},
\end{align*}
then by \eqref{eq-3.8-Hw large}, we can choose $s_0 < c$ to obtain
\begin{align*}
	\bP \left( \exists w' \in B(w, 2r_H): \|Hw'\| \le s_0/2, \|X\| \le C_0 \right)
	&\le \bP \left( \|Hw\| \le s_0, \|X\| \le C_0 \right) \\
	&\le \exp (-cn),
\end{align*}

Note that for small $\theta>0$ that will be determined in the sequel, and $I \subseteq [N+n-k]$ with $|I| = \theta (N+n-k)$, by Lemma \ref{Lemma-Compressible}, the set of unit vector in $\bS^{N+n-k-1}$ supported in $I$ has a $r_H$-net of cardinality bounded by $(3/r_H)^{2|I|}$. Thus,
\begin{align*}
	& \bP \left( \exists w' \in \Comp(\theta,r_H): \|Hw'\| \le s_0/2, \|X\| \le C_0 \right) \\
	\le& \sum_{I \subseteq [N+n-k], |I| = \theta (N+n-k)} \bP \left( \exists w' \in \bS^{2n-1} \cap B(\bS^{2n-1}_I,r_H): \|Hw'\| \le s_0/2, \|X\| \le C_0 \right) \\
	\le& \binom{N+n-k}{\theta (N+n-k)} \left( \dfrac{3}{r_H} \right)^{2\theta (N+n-k)} \bP \left( \exists w' \in B(w, 2r_H): \|Hw'\| \le s_0/2, \|X\| \le C_0 \right) \\
	\le& \binom{N+n-k}{\theta (N+n-k)} \left( \dfrac{3}{r_H} \right)^{2\theta (N+n-k)} \exp (-cn).
\end{align*}
By the Stirling formula,
\begin{align*}
	\binom{N+n-k}{\theta (N+n-k)}
	&= \dfrac{(N+n-k) \cdots (N+n-k - \theta (N+n-k) + 1)}{(\theta (N+n-k))!} \\
	&\le \dfrac{(N+n-k)^{\theta (N+n-k)}}{(\theta (N+n-k))!} \\
	&\sim \dfrac{(N+n-k)^{\theta (N+n-k)} e^{\theta (N+n-k)}}{\sqrt{2 \pi \theta (N+n-k)}(\theta (N+n-k))^{\theta (N+n-k)}} \\
	&= \dfrac{e^{\theta (N+n-k)}}{\sqrt{2 \pi \theta (N+n-k)}\theta^{\theta (N+n-k)}}.
\end{align*}
Hence, when $n$ large,
\begin{align*}
	& \bP \left( \exists w' \in \Comp(\theta,r_H): \|Hw'\| \le s_0/2, \|X\| \le C_0 \right) \\
	\le& \left( \dfrac{e}{\theta} \right)^{\theta (N+n-k)} \left( \dfrac{3}{r_H} \right)^{2\theta (N+n-k)} \exp (-cn) \\
	=& \exp \left( -cn + \theta (N+n-k) \ln \left( \dfrac{9e}{\theta r_H^2} \right) \right).
\end{align*}
Since $\theta \ln ((9e)/(\theta r_H^2))$ tends to zero as $\theta$ tends to zero, we can choose $\theta = \theta_0$, where $\theta_0$ is sufficiently small, such that
\begin{align*}
	\bP \left( \exists w' \in \Comp(\theta_0,r_H): \|Hw'\| \le s_0/2, \|X\| \le C_0 \right)
	\le \exp \left( -cn \right),
\end{align*}
which lead to \eqref{eq-3.4-Compressible}.

The proof of \eqref{eq-3.4-Compressible} for the case $2k+1>n$ is similar and is omitted.

\subsection{Estimate \texorpdfstring{\eqref{eq-3.5-Incompressible}}{(3.5)} for imcompressible vectors for the case \texorpdfstring{$2k+1 \le n$}{2k+1 <= n}}\label{step2}

We now establish the estimate \eqref{eq-3.5-Incompressible} for the case $2k+1 \le n$ with $\theta = \theta_0$ and $\rho = r_H$. By Lemma \ref{Lemma-Incompressible distance}, it is enough to prove
\begin{align} \label{eq-3.9-distance prob}
	\bP \left( \dist(H_{[N+n-k],l}, H_{-l}) \le n^{-13/11}, \|X\| \le C_0 \right) \le C n^{-1/22}, \ \forall l \in [N+n-k].
\end{align}
By Lemma \ref{Lemma-distant}, we have
\begin{align} \label{eq-3.10-distance}
	\dist(H_{[N+n-k],l}, H_{-l}) = \dfrac{\Num}{\Den},
\end{align}
where
\begin{align} \label{eq-3.11-Def Num}
	\Num = \left| H_{ll} - H_{l,[N+n-k] \setminus \{l\}} \left( H_{[N+n-k] \setminus \{l\}, [N+n-k] \setminus \{l\}} \right)^{-1} H_{[N+n-k] \setminus \{l\}, l} \right|,
\end{align}
and
\begin{align} \label{eq-3.12-Def Den}
	\Den = \sqrt{1 + \left\| H_{l,[N+n-k] \setminus \{l\}} \left( H_{[N+n-k] \setminus \{l\}, [N+n-k] \setminus \{l\}} \right)^{-1} \right\|^2}.
\end{align}

Next, we compute \eqref{eq-3.10-distance} for the three cases $N+k+1 \le l \le N+n-k$, $N+1 \le l \le N+k$ and $1 \le l \le N$ individually.

\subsubsection{Case of \texorpdfstring{$N+k+1 \le l \le N+n-k$}{N+k+1 <= l <= N+n-k}}\label{Case1}

We estimate \eqref{eq-3.10-distance} for the case $l = N+n-k$ first. Recalled the definition of $H$, we have
\begin{align} \label{eq-3.13-entries of H}
	&H_{N+n-k,N+n-k} = 0, \
	H_{N+n-k,[N+n-k-1]} = \left( X_{n-k}^*, \left( e_k^{(n-k-1)} \right)^\intercal \right), \nonumber \\
	&H_{[N+n-k-1],N+n-k} = \left(
	\begin{matrix}
	X_{n-k} \\
	e_{n-2k}^{(n-k-1)} \\
	\end{matrix}
	\right), \nonumber \\
	& H_{[N+n-k-1],[N+n-k-1]} = \left(
	\begin{matrix}
	zI_N & Y^{(1)} \\
	Y^{(2)} & B
	\end{matrix}
	\right) \nonumber \\
	&= \left(
	\begin{matrix}
	zI_N & \left( X_{n-k+1}, \ldots, X_n, X_{k+1}, \ldots, X_{n-k-1} \right) \\
	\left( X_1, \ldots, X_{n-k-1} \right)^* & \left( e_{n-2k+1}^{(n-k-1)}, \ldots, e_{n-k-1}^{(n-k-1)}, 0, e_1^{(n-k-1)}, \ldots, e_{n-2k-1}^{(n-k-1)} \right)
	\end{matrix}
	\right).
\end{align}

We first show that $H_{[N+n-k-1],[N+n-k-1]}$ is invertible almost surely. Apply the row operation to the determinant, we can see that
\begin{align*}
	\det \left( H_{[N+n-k-1],[N+n-k-1]} \right)
	= \det \left( B - z^{-1} Y^{(2)} Y^{(1)} \right).
\end{align*}
By a similar argument above, we can show that the determinant is a non-zero polynomial of the entries of $X$. Since the entries of $X$ have density, the determinant vanishes with probability zero.

Denote
\begin{align*}
	\left( H_{[N+n-k-1],[N+n-k-1]} \right)^{-1}
	= \left(
	\begin{matrix}
	D & E \\
	F & G
	\end{matrix}
	\right),
\end{align*}
where $D \in \bC^{N \times N}$, $G \in \bC^{(n-k-1) \times (n-k-1)}$.

Next, we compute the $\Num$ given by \eqref{eq-3.11-Def Num} and $\Den$ given by \eqref{eq-3.12-Def Den} individually.

\bigskip
\noindent\textbf{Step (a).} We consider $\Num$ first. Let $\xi = \left\{ \xi_1, \ldots, \xi_N \right\}$ be i.i.d. Bernoulli random variables with $\bP(\xi_1 = 1) = p$, where $p \in (0,1)$ will be determined in the sequel. Moreover, we choose these variables to be independent of everything else. Set $I = \{i \in [N]: \xi_i = 1\}$. Choose three random vectors $x,x',x'' \in \bC^{N}$, such that their entries are independent each other and independent of everything else, and that $x,x',x'' \overset{d}{=} X_{n-k}$. Set
\begin{align*}
	u = (x)_{I}, \
	v = (x')_{I^\complement}, \
	w = (x'')_{I^\complement}.
\end{align*}
Denote $\widetilde{X} = (X_1, \ldots, X_{n-k-1}, X_{n-k+1}, \ldots, X_n)$, then by \eqref{eq-3.11-Def Num}, \eqref{eq-3.13-entries of H}, Cauchy-Schwarz inequality and Lemma \ref{Lemma-quadratic form bound}, we have
\begin{align} \label{eq-3.14-expression by E and cS}
	& \bP \left( \Num \le t, \|X\| \le C_0 \right)^2 \nonumber \\
	\le& \bP \left( \Num \le t, \left\| \widetilde{X} \right\| \le C_0 \right)^2 \nonumber \\
	=& \bP \left( \left\| \widetilde{X} \right\| \le C_0, \right. \nonumber \\
	& \left. \left| X_{n-k}^* D X_{n-k} + \left(e_k^{(n-k-1)} \right)^\intercal F X_{n-k} + X_{n-k}^* E e_{n-2k}^{(n-k-1)} + \left(e_k^{(n-k-1)} \right)^\intercal G e_{n-2k}^{(n-k-1)} \right| \le t \right)^2 \nonumber \\
	=& \left( \bE_{\widetilde{X}} \left[ \bE_{X_{n-k}} \left[ 1_{\left| X_{n-k}^* D X_{n-k} + \left(e_k^{(n-k-1)} \right)^\intercal F X_{n-k} + X_{n-k}^* E e_{n-2k}^{(n-k-1)} + \left(e_k^{(n-k-1)} \right)^\intercal G e_{n-2k}^{(n-k-1)} \right| \le t} \right] 1_{\left\| \widetilde{X} \right\| \le C_0} \right] \right)^2 \nonumber \\
	\le& \bE_{\widetilde{X}} \left[ \left( \bE_{X_{n-k}} \left[ 1_{\left| X_{n-k}^* D X_{n-k} + \left(e_k^{(n-k-1)} \right)^\intercal F X_{n-k} + X_{n-k}^* E e_{n-2k}^{(n-k-1)} + \left(e_k^{(n-k-1)} \right)^\intercal G e_{n-2k}^{(n-k-1)} \right| \le t} \right] \right)^2 1_{\left\| \widetilde{X} \right\| \le C_0} \right] \nonumber \\
	\le& \bE_{\widetilde{X}} \left[ \bE_{v,w} \left[ \cS_u \left( (v-w)^* D_{I^\complement,I} u + u^* D_{I,I^\complement} (v-w), 2t \right) \right] 1_{\|\widetilde{X}\| \le C_0} \right] \nonumber \\
	=& \bE_{\widetilde{X}} \left[ \bE_{v,w} \left[ \cS_u \left( (\Pi_{I^\complement}(x'-x''))^* D \left( \Pi_Ix \right) + (\Pi_Ix)^* D \left( \Pi_{I^\complement}(x'-x'') \right), 2t \right) 1_{\|\widetilde{X}\| \le C_0} \right] \right].
\end{align}
Denote
\begin{align*}
	&y = \dfrac{D \left( \Pi_{I^\complement}(x'-x'') \right)}{\left\| D \left( \Pi_{I^\complement}(x'-x'') \right) \right\|}, \ 
	\alpha = \dfrac{\sqrt{n} \left\| D \left( \Pi_{I^\complement}(x'-x'') \right) \right\|}{\left\| D \right\|_{HS}}, \\
	&\tilde{y}^* = \dfrac{\left( \Pi_{I^\complement}(x'-x'') \right)^* D}{\left\| \left( \Pi_{I^\complement}(x'-x'') \right)^* D \right\|}, \ 
	\tilde{\alpha} = \dfrac{\sqrt{n} \left\| \left( \Pi_{I^\complement}(x'-x'') \right)^* D \right\|}{\left\| D \right\|_{HS}}.
\end{align*}
Here, we use the convention that $y=\tilde{y} = 0$ and $\alpha = \tilde{\alpha} = 0$ if $I = [N]$. Let $W_i = \tilde{\alpha} \overline{\tilde{y}}_i x_i + \alpha y_i \overline{x_i}$ for $i \in I$. When conditioning on $v,w$ and $\widetilde{X}$, $\{W_i: i \in I\}$ are independent. Besides, we have
\begin{align} \label{eq-3.15-simplify cS}
	(\Pi_{I^\complement}(x'-x''))^* D \left( \Pi_Ix \right) + (\Pi_Ix)^* D \left( \Pi_{I^\complement}(x'-x'') \right)
	&= \dfrac{\left\| D \right\|_{HS}}{\sqrt{n}} (\tilde{\alpha} \tilde{y}^* \left( \Pi_Ix \right) + (\Pi_I x)^* y \alpha) \nonumber \\
	&= \dfrac{\left\| D \right\|_{HS}}{\sqrt{n}} \sum_{i \in I} W_i.
\end{align}
Hence, by \eqref{eq-3.14-expression by E and cS}, \eqref{eq-3.15-simplify cS}, Definition \ref{Def-small ball prob} and Lemma \ref{Lemma-small ball prob monotone},
\begin{align} \label{eq-3.16-Num prob simplify}
	\bP \left( \Num \le t, \|X\| \le C_0 \right)^2
	&\le \bE_{\widetilde{X}} \left[ \bE_{v,w} \left[ \cS_u \left( \dfrac{\left\| D \right\|_{HS}}{\sqrt{n}} \sum_{i \in I} W_i, 2t \right) 1_{\left\| \widetilde{X} \right\| \le C_0} \right] \right] \nonumber \\
	&= \bE_{\widetilde{X}} \left[ \bE_{v,w} \left[ \cS_u \left( \sum_{i \in I} W_i, \dfrac{2\sqrt{n}t}{\left\| D \right\|_{HS}} \right) 1_{\left\| \widetilde{X} \right\| \le C_0} \right] \right] \nonumber \\
	&\le \bE_{\widetilde{X}} \left[ \bE_{v,w} \left[ \cS_u \left( \sum_{i \in I \cap J} W_i, \dfrac{2\sqrt{n}t}{\left\| D \right\|_{HS}} \right) 1_{\left\| \widetilde{X} \right\| \le C_0} \right] \right],
\end{align}
where the index set $J$ is given by
\begin{align} \label{eq-def-J}
	J = \left\{ j \in [N]: \dfrac{r_H}{\sqrt{N}} \le |y_j| \le \dfrac{2}{\sqrt{\theta_0N}}, |\tilde{y}_j| \le \dfrac{2}{\sqrt{\theta_0N}} \right\}.
\end{align}
Note that $x_i$ has the same distribution as $X_{11}^{(n)}$, by conditions (C1), (C2) and Cauchy-Schwarz inequality, we have
\begin{align} \label{eq-3.17-Zi 2nd moment}
	\sum_{i \in I \cap J} \bE_u \left[ |W_i|^2 \right]
	&= \sum_{i \in I \cap J} \left( \left( \tilde{\alpha}^2 \left| \tilde{y}_i \right|^2 + \alpha^2 |y_i|^2 \right) \bE \left[ \left| x_i \right|^2 \right] + 2 \alpha \tilde{\alpha} \Re \left( \bE \left[ x_i^2 \right] \overline{\tilde{y}_i} \overline{y_i} \right) \right) \nonumber \\
	&= \sum_{i \in I \cap J} \left( \left( \tilde{\alpha}^2 \left| \tilde{y}_i \right|^2 + \alpha^2 |y_i|^2 \right) \bE \left[ \left| X_{11}^{(n)} \right|^2 \right] + 2 \alpha \tilde{\alpha} \Re \left( \bE \left[ \left( X_{11}^{(n)} \right)^2 \right] \overline{\tilde{y}_i} \overline{y_i} \right) \right) \nonumber \\
	&\ge \sum_{i \in I \cap J} \left( \left( \tilde{\alpha}^2 \left| \tilde{y}_i \right|^2 + \alpha^2 |y_i|^2 \right) \bE \left[ \left| X_{11}^{(n)} \right|^2 \right] - 2 \alpha \tilde{\alpha} \left| \bE \left[ \left( X_{11}^{(n)} \right)^2 \right] \right| \left| \tilde{y}_i \right| \left| y_i \right| \right) \nonumber \\
	&\ge \sum_{i \in I \cap J} \left( \tilde{\alpha}^2 \left| \tilde{y}_i \right|^2 + \alpha^2 |y_i|^2 \right) \left( \dfrac{1}{n} - \left| \bE \left[ \left( X_{11}^{(n)} \right)^2 \right] \right| \right) \nonumber \\
	&\ge \dfrac{c_0}{nN} r_H^2 \alpha^2 |I \cap J|,
\end{align}
and
\begin{align} \label{eq-3.18-Zi 3rd moment}
	\sum_{i \in I \cap J} \bE_u \left[ |W_i|^3 \right]
	&\le 4 \sum_{i \in I \cap J} \bE_u \left[ |\tilde{\alpha} \overline{\tilde{y}}_i u_i|^3 + |\alpha y_i \overline{u_i}|^3 \right] \nonumber \\
	&\le \dfrac{32}{(\theta_0 N)^{3/2}} \sum_{i \in I \cap J} \bE_u \left[ \tilde{\alpha}^3 |u_i|^3 + \alpha^3 |\overline{u_i}|^3 \right] \nonumber \\
	&= \dfrac{32}{(\theta_0 N)^{3/2}} (\alpha^3 + \tilde{\alpha}^3) \bE \left[ \left| X_{11}^{(n)} \right|^3 \right] |I \cap J| \nonumber \\
	&\le \dfrac{32}{(\theta_0 N)^{3/2}} (\alpha^3 + \tilde{\alpha}^3) \left( \bE \left[ \left| X_{11}^{(n)} \right|^4 \right] \right)^{3/4} |I \cap J| \nonumber \\
	&\le \dfrac{C}{(\theta_0 N n)^{3/2}} (\alpha^3 + \tilde{\alpha}^3) |I \cap J|.
\end{align}
Here, $C$ is a large constant. Hence, by \eqref{eq-3.17-Zi 2nd moment}, \eqref{eq-3.18-Zi 3rd moment} and Lemma \ref{Lemma-small ball prob-upper bound},
\begin{align} \label{eq-3.16-upper bound of cS}
	\cS_u \left( \sum_{i \in I \cap J} W_i, \dfrac{2\sqrt{n}t}{\left\| D \right\|_{HS}} \right)
	\le& \dfrac{c'}{\sqrt{\sum_{i \in I \cap J} \bE_u [|W_i|^2]}} \dfrac{2\sqrt{n}t}{\left\| D \right\|_{HS}} + \dfrac{c'\sum_{i \in I \cap J} \bE_u [|W_i|^3]}{\left( \sum_{i \in I \cap J} \bE_u [|W_i|^2] \right)^{3/2}} \nonumber \\
	\le& \dfrac{C n \sqrt{N} t}{r_H \alpha \sqrt{|I \cap J|} \left\| D \right\|_{HS}} + \dfrac{C (\alpha^3 + \tilde{\alpha}^3)}{\theta_0^{3/2} r_H^3 \alpha^3 \sqrt{|I \cap J|}}.
\end{align}
Here, we use the convention that the right hand side of \eqref{eq-3.16-upper bound of cS} is $1$ if $|I \cap J| = \emptyset$.

Next, we estimate the lower bound of $|I \cap J|$. Take $p = 1 - \theta_0/4$ and set the event
\begin{align*}
	\cE_I = \left\{ |I| > N (1-\theta_0/3) \right\}
	= \left\{ \sum_{i=1}^N \xi_i > N (1-\theta_0/3) \right\},
\end{align*}
then by Hoeffding concentration inequality (Lemma \ref{Lemma-Hoeffding concentration}),
\begin{align} \label{eq-3.17-prob large |I|}
	\bP \left( \cE_I^\complement \right)
	= \bP \left( \sum_{i=1}^N \xi_i \le N (1-\theta_0/3) \right)
	= \bP \left( \sum_{i=1}^N \left( \xi_i - \bE [\xi_i] \right) \le -N \theta_0/12 \right)
	\le \exp (- c N \theta_0^2).
\end{align}

Now we estimate $|J|$. To do this, we need the following estimation on the matrix $D$ from Lemma~\ref{Lemma-similar Coro 14} (established later):
for any deterministic vector $d \in \bC^N \setminus \{0\}$,
\begin{align*}
	\bP \left( \dfrac{D d}{\left\| D d \right\|} \in \Comp(\theta_0,r_H), \left\| \widetilde{X} \right\| \le C_0 \right) \le \exp (-cN).
\end{align*}

Since, the entries of $x'$ and $x''$ are independent and have the same distribution as $X_{11}^{(n)}$, which has continuous density. If $I^\complement \not= \emptyset$, then $\Pi_{I^\complement}(x'-x'') = 0$ with probability zero. Denote
\begin{align*}
	\cE_{y,\Incomp} = \left\{ y \in \Incomp(\theta_0,r_H) \right\},
\end{align*}
then by Lemma \ref{Lemma-similar Coro 14},
\begin{align} \label{eq-3.23-prob cE Incomp}
	& \bP \left( \cE_{y,\Incomp}^\complement \cap \left\{ \left\| \widetilde{X} \right\| \le C_0 \right\} \right) \nonumber \\
	\le& \bP \left( I = [N] \right) + \bP \left( \dfrac{D \left( \Pi_{I^\complement}(x'-x'') \right)}{\left\| D \left( \Pi_{I^\complement}(x'-x'') \right) \right\|} \in \Comp(\theta_0, r_H), \Pi_{I^\complement}(x'-x'') \not= 0 \right) \nonumber \\
	=& \bP \left( I = [N] \right) + \bE_{x',x''} \left[\bP \left( \left. \dfrac{D \left( \Pi_{I^\complement}(x'-x'') \right)}{\left\| D \left( \Pi_{I^\complement}(x'-x'') \right) \right\|} \in \Comp(\theta_0, r_H) \right| x',x''\right) 1_{\Pi_{I^\complement}(x'-x'') \not= 0} \right] \nonumber \\
	\le& \bP \left( I = [N] \right) + \exp(-cN) \nonumber \\
	=& p^N + \exp (-cN) \nonumber \\
	\le& \exp (-cN).
\end{align}
On $\cE_{y,\Incomp}$, by Lemma \ref{Lemma-Incompressible}, $|J| \ge \theta_0 N/2$. Thus, by \eqref{eq-3.17-prob large |I|}, on the event $\cE_{y,\Incomp} \cap \cE_I$,
\begin{align} \label{eq-3.24-|I cap J| lower bound}
	|I \cap J|
	\ge |I| + |J| - N
	> N(1-\theta_0/3) + \theta_0N/2 - N
	= \theta_0N/6.
\end{align}

Next, we estimate $(\alpha^3 + \tilde{\alpha}^3) /\alpha^3$ and $\alpha^{-1}$ in \eqref{eq-3.16-upper bound of cS}. Denote $\cE_{\alpha} = \left\{ \gamma \le \alpha \le \beta^{-1}, \tilde{\alpha} \le \beta^{-1} \right\}$, where $\beta$ and $\gamma$ are $o(1)$ and will be determined in the sequel. Then on the event $\cE_{\alpha}$, we have
\begin{align} \label{eq-3.25-estimation on alpha}
	\dfrac{\alpha^3 + \tilde{\alpha}^3}{\alpha^3} \le 2 \beta^{-3} \gamma^{-3}, \
	\alpha^{-1} \le \gamma^{-1}.
\end{align}

Next, we show that the event $\cE_{\alpha}$ has high probability. Indeed, we have
\begin{align} \label{eq-3.26-cE alpha}
	\bP \left( \cE_{\alpha}^\complement, \left\| \widetilde{X} \right\| \le C_0 \right)
	\le \bP \left( \alpha < \gamma, \left\| \widetilde{X} \right\| \le C_0 \right)
	+ \bP \left( \alpha > \beta^{-1} \right)
	+ \bP \left( \tilde{\alpha} > \beta^{-1} \right).
\end{align}
Recalled that $\left( \Pi_{I^\complement}(x'-x'') \right)_i = (1-\xi_i) (x' - x'')_i$, so the entries of $\Pi_{I^\complement}(x'-x'')$ are i.i.d. with mean zero and variance $2(1-p)/n$. Thus, by Markov inequality, the second term of the right hand side of \eqref{eq-3.26-cE alpha} is
\begin{align} \label{eq-3.27-2rd of 3.26}
	& \bP \left( \alpha > \beta^{-1} \right) \nonumber \\
	=& \bE_{\widetilde{X}} \left[ \bP \left( \alpha > \beta^{-1} | \widetilde{X} \right) \right] \nonumber \\
	=& \bE_{\widetilde{X}} \left[ \bP \left( \left. \left\| D \left( \Pi_{I^\complement}(x'-x'') \right) \right\| > \dfrac{\left\| D \right\|_{HS}}{\sqrt{n} \beta} \right| \widetilde{X} \right) \right] \nonumber \\
	\le& n \beta^2 \bE_{\widetilde{X}} \left[ \left\| D \right\|_{HS}^{-2} \bE \left[ \left. \left\| D \left( \Pi_{I^\complement}(x'-x'') \right) \right\|^2 \right| \widetilde{X} \right] \right] \nonumber \\
	=& n \beta^2 \bE_{\widetilde{X}} \left[ \left\| D \right\|_{HS}^{-2} \bE \left[ \left. \sum_{i_1,i_2,i_3=1}^n \overline{(\Pi_{I^\complement}(x'-x''))}_{i_1} D^*_{i_1i_2} D_{i_2i_3} (\Pi_{I^\complement}(x'-x''))_{i_3} \right| \widetilde{X} \right] \right] \nonumber \\
	=& 2(1-p) \beta^2 \bE_{\widetilde{X}} \left[ \left\| D \right\|_{HS}^{-2} \sum_{i_1,i_2=1}^n D^*_{i_1i_2} D_{i_2i_1} \right] \nonumber \\
	=& 2(1-p) \beta^2.
\end{align}
By the same argument, the third term of the right hand side of \eqref{eq-3.26-cE alpha} is also bounded by $2(1-p) \beta^2$. Next, we deal with the first term of \eqref{eq-3.26-cE alpha}. We denote
\begin{align*}
	\left( \tilde{u}^{(i)} \right)^* = \dfrac{ \left( e_i^{(N)} \right)^* D}{\left\| \left( e_i^{(N)} \right)^* D \right\|}, i \in [N].
\end{align*}
Then we have
\begin{align} \label{eq-3.35-entrywise}
	\left\|D \left( \Pi_{I^\complement}(x'-x'') \right) \right\|^2
	=& \sum_{i=1}^N \left| \left( e_i^{(N)} \right)^* D \left( \Pi_{I^\complement}(x'-x'') \right) \right|^2 \nonumber \\
	=& \sum_{i=1}^N \left\| \left( e_i^{(N)} \right)^* D \right\|^2 \left| \left( \tilde{u}^{(i)} \right)^* \left( \Pi_{I^\complement}(x'-x'') \right) \right|^2.
\end{align}
Denote
\begin{align*}
	I_i = \left\{ j \in [N]: \dfrac{r_H}{\sqrt{N}} \le |\tilde{u}^{(i)}_j| \le \dfrac{2}{\sqrt{\theta_0 N}} \right\}.
\end{align*}
Note that
\begin{align*}
	\sum_{i=1}^n \left\| \left( e_i^{(N)} \right)^* D \right\|^2 = \left\| D \right\|_{HS}^2,
\end{align*}
by Lemma \ref{Lemma-probability of weight rv}, Lemma \ref{Lemma-small ball prob monotone} and Lemma \ref{Lemma-small ball prob-upper bound}, we have
\begin{align} \label{eq-3.27'}
	& \bP_{v,w,\xi} \left( \sum_{i=1}^N \left\| \left( e_i^{(N)} \right)^* D \right\|^2 \left| \left( \tilde{u}^{(i)} \right)^* \left( \Pi_{I^\complement}(x'-x'') \right) \right|^2 \le \dfrac{\gamma^2}{n} \left\| D \right\|_{HS}^2 \right) \nonumber \\
	\le& 2 \sum_{i=1}^N \dfrac{\left\| \left( e_i^{(N)} \right)^* D \right\|^2}{\left\| D \right\|_{HS}^2} \bP_{v,w,\xi} 
	\left( \left| \left( \tilde{u}^{(i)} \right)^* \left( \Pi_{I^\complement}(x'-x'') \right) \right|^2 \le \dfrac{2\gamma^2}{n} \right) \nonumber \\
	\le& 2 \sum_{i=1}^N \dfrac{\left\| \left( e_i^{(N)} \right)^* D \right\|^2}{\left\| D \right\|_{HS}^2} \cS_{v,w,\xi} \left( \sum_{j \in [N]} \overline{\tilde{u}^{(i)}_j} \left( \Pi_{I^\complement}(x'-x'') \right)_j, \dfrac{\sqrt{2} \gamma}{\sqrt{n}} \right) \nonumber \\
	\le& 2 \sum_{i=1}^N \dfrac{\left\| \left( e_i^{(N)} \right)^* D \right\|^2}{\left\| D \right\|_{HS}^2} \cS_{v,w,\xi} \left( \sum_{j \in I_i} \overline{\tilde{u}^{(i)}_j} \left( \Pi_{I^\complement}(x'-x'') \right)_j, \dfrac{\sqrt{2} \gamma}{\sqrt{n}} \right) \nonumber \\
	\le& 2 \sum_{i=1}^N \dfrac{\left\| \left( e_i^{(N)} \right)^* D \right\|^2}{\left\| D \right\|_{HS}^2} \left( \dfrac{C \sqrt{N} \gamma}{\sqrt{1-p} r_H \sqrt{|I_i|}} + \dfrac{C}{\sqrt{1-p} \theta_0^{3/2} r_H^3 \sqrt{|I_i|}} \right) \nonumber \\
	\le& 2 \left( \dfrac{C \sqrt{N} \gamma}{\sqrt{1-p} r_H} + \dfrac{C}{\sqrt{1-p} \theta_0^{3/2} r_H^3} \right) \max_{i \in [N]} \dfrac{1}{\sqrt{|I_i|}}.
\end{align}
Let
\begin{align*}
	\cE_{u,\Incomp} = \left\{ \tilde{u}^{(i)} \in \Incomp(\theta_0, r_H), \forall i \in [N] \right\}.
\end{align*}
Then by Lemma \ref{Lemma-Incompressible}, on $\cE_{u,\Incomp}$, $|I_i| \ge \theta_0 N/2$ for all $i \in [N]$. Moreover, by Remark \ref{Remark--similar Coro 14}, $\bP \left( \cE_{u,\Incomp}^\complement, \left\| \widetilde{X} \right\| \le C_0 \right) \le N \exp(-cN)$.
Thus, by \eqref{eq-3.35-entrywise} and \eqref{eq-3.27'}, the first term of \eqref{eq-3.26-cE alpha} is
\begin{align} \label{eq-3.28-1st of 3.26}
	& \bP \left( \alpha < \gamma, \left\| \widetilde{X} \right\| \le C_0 \right) \nonumber \\
	=& \bP \left( \left\|D \left( \Pi_{I^\complement}(x'-x'') \right) \right\| < \dfrac{\gamma}{\sqrt{n}} \left\| D \right\|_{HS}, \left\| \widetilde{X} \right\| \le C_0 \right) \nonumber \\
	=& \bP \left( \sum_{i=1}^N \left\| \left( e_i^{(N)} \right)^* D \right\|^2 \left| \left( \tilde{u}^{(i)} \right)^* \left( \Pi_{I^\complement}(x'-x'') \right) \right|^2 < \dfrac{\gamma^2}{n} \left\| D \right\|_{HS}^2, \left\| \widetilde{X} \right\| \le C_0 \right) \nonumber \\
	=& \bE_{\widetilde{X}} \left[ \bP \left( \left. \sum_{i=1}^N \left\| \left( e_i^{(N)} \right)^* D \right\|^2 \left| \left( \tilde{u}^{(i)} \right)^* \left( \Pi_{I^\complement}(x'-x'') \right) \right|^2 < \dfrac{\gamma^2}{n} \left\| D \right\|_{HS}^2 \right| \widetilde{X} \right) 1_{\left\| \widetilde{X} \right\| \le C_0} \right] \nonumber \\
	=& \bE_{\widetilde{X}} \left[ \bP \left( \left. \sum_{i=1}^N \left\| \left( e_i^{(N)} \right)^* D \right\|^2 \left| \left( \tilde{u}^{(i)} \right)^* \left( \Pi_{I^\complement}(x'-x'') \right) \right|^2 < \dfrac{\gamma^2}{n} \left\| D \right\|_{HS}^2 \right| \widetilde{X} \right) 1_{\cE_{u,\Incomp}} 1_{\left\| \widetilde{X} \right\| \le C_0} \right] \nonumber \\
	& +\bE_{\widetilde{X}} \left[ \bP \left( \left. \sum_{i=1}^N \left\| \left( e_i^{(N)} \right)^* D \right\|^2 \left| \left( \tilde{u}^{(i)} \right)^* \left( \Pi_{I^\complement}(x'-x'') \right) \right|^2 < \dfrac{\gamma^2}{n} \left\| D \right\|_{HS}^2 \right| \widetilde{X} \right) 1_{\cE_{u,\Incomp}^\complement} 1_{\left\| \widetilde{X} \right\| \le C_0} \right] \nonumber \\
	\le& 2 \left( \dfrac{C \sqrt{N} \gamma}{\sqrt{1-p} r_H} + \dfrac{C}{\sqrt{1-p} \theta_0^{3/2} r_H^3} \right) \dfrac{\sqrt{2}}{\sqrt{\theta_0 N}} + \bP \left( \cE_{u,\Incomp}^\complement, \left\| \widetilde{X} \right\| \le C_0 \right) \nonumber \\
	\le& \dfrac{C \gamma}{\sqrt{(1-p)\theta_0}r_H} + \dfrac{C}{\sqrt{(1-p)} \theta_0^2 r_H^3 \sqrt{N}} + N \exp(-cN).
\end{align}
Therefore, by \eqref{eq-3.26-cE alpha}, \eqref{eq-3.27-2rd of 3.26} and \eqref{eq-3.28-1st of 3.26},
\begin{align} \label{eq-3.29-prob cE alpha}
	\bP \left( \cE_{\alpha}^\complement, \left\| \widetilde{X} \right\| \le C_0 \right)
	\le 4(1-p) \beta^2 + \dfrac{C \gamma}{\sqrt{(1-p)\theta_0}r_H} + \dfrac{C}{\sqrt{(1-p)} \theta_0^2 r_H^3 \sqrt{N}} + N \exp(-cN).
\end{align}

Lastly, by \eqref{eq-3.16-Num prob simplify}, \eqref{eq-3.16-upper bound of cS}, \eqref{eq-3.17-prob large |I|}, \eqref{eq-3.23-prob cE Incomp}, \eqref{eq-3.24-|I cap J| lower bound}, \eqref{eq-3.25-estimation on alpha} and \eqref{eq-3.29-prob cE alpha},
\begin{align*}
	& \bP \left( \Num \le s \left\| D \right\|_{HS}, \|X\| \le C_0 \right)^2 \\
	\le& \bE_{\widetilde{X}} \left[ \bE_{v,w} \left[ \cS_u \left( \sum_{i \in I \cap J} W_i, 2\sqrt{n}s \right) 1_{\|\widetilde{X}\| \le C_0} \right] \right] \\
	\le& \bE_{\widetilde{X}} \left[ \bE_{v,w} \left[ \cS_u \left( \sum_{i \in I \cap J} W_i, 2\sqrt{n}s \right) 1_{\{\|\widetilde{X}\| \le C_0\} \cap \cE_{y,\Incomp} \cap \cE_I \cap \cE_{\alpha}} \right] \right] \nonumber \\
	&+ \bE_{\widetilde{X}} \left[ \bE_{v,w} \left[ 1_{\{\left\| \widetilde{X} \right\| \le C_0\} \cup \cE_{y,\Incomp}^\complement \cup \cE_I^\complement \cup \cE_{\alpha}^\complement} \right] \right] \\
	\le& \bE_{\widetilde{X}} \left[ \bE_{v,w} \left[ \left( \dfrac{C n\sqrt{N} s}{r_H \alpha \sqrt{|I \cap J|}} + \dfrac{C (\alpha^3 + \tilde{\alpha}^3)}{\theta_0^{3/2} r_H^3 \alpha^3 \sqrt{|I \cap J|}} \right) 1_{\{\|\widetilde{X}\| \le C_0\} \cap \cE_{y,\Incomp} \cap \cE_I \cap \cE_{\alpha}} \right] \right] \\
	& + \bP \left( \left\{ \left\| \widetilde{X} \right\| \le C_0 \right\} \cup \cE_{y,\Incomp}^\complement \cup \cE_I^\complement \cup \cE_{\alpha}^\complement \right) \\
	\le& \dfrac{C n s}{r_H \gamma \sqrt{\theta_0}} + \dfrac{C \beta^{-3} \gamma^{-3}}{\theta_0^2 r_H^3 \sqrt{N}} + \bP \left( \left\{ \left\| \widetilde{X} \right\| \le C_0 \right\} \cup \cE_{y,\Incomp}^\complement \cup \cE_I^\complement \cup \cE_{\alpha}^\complement \right) \\
	\le& \dfrac{C n s}{r_H \gamma \sqrt{\theta_0}} + \dfrac{C \beta^{-3} \gamma^{-3}}{\theta_0^2 r_H^3 \sqrt{N}} + \bP \left( \left\{ \left\| \widetilde{X} \right\| \le C_0 \right\} \cup \cE_{y,\Incomp}^\complement \right) + \bP \left( \cE_I^\complement \right) + \bP \left( \left\{ \left\| \widetilde{X} \right\| \le C_0 \right\} \cup \cE_{\alpha}^\complement \right) \\
	\le& C(\theta_0, r_H) \dfrac{ns}{\gamma} + \dfrac{C(\theta_0, r_H)}{\beta^3 \gamma^3 \sqrt{N}} + \exp (- c N) + \exp \left( -cN \theta_0^2 \right) + C(\theta_0) \beta^2 + C(\theta_0, r_H) \gamma \nonumber \\
	& + \dfrac{C(\theta_0, r_H)}{\sqrt{N}} + N \exp (-cN) \nonumber \\
	\le& C(\theta_0,r_H) \left( \dfrac{ns}{\gamma} + \dfrac{1}{\beta^3 \gamma^3 \sqrt{N}} + \beta^2 + \gamma + \dfrac{1}{\sqrt{N}} \right).
\end{align*}
Here, $C(\theta_0, r_H)$ is a large positive constant that may depend on $\theta_0, r_H$ (and $z$). Then we may choose $\beta  = n^{-1/22}$ and $\gamma = n^{-1/11}$ to obtain
\begin{align} \label{eq-3.30-Prob on Num}
	\bP \left( \Num \le s \left\| D \right\|_{HS}, \|X\| \le C_0 \right)^2
	\le C(\theta_0, r_H) \left( n^{12/11} s + n^{-1/11} \right).
\end{align}

\medskip
\noindent\textbf{Step (b).} We compute $\Den$ given by \eqref{eq-3.12-Def Den}. Recalled the definition of $H$ in \eqref{eq-3.13-entries of H} and the blocking of $\left( H_{[N+n-k-1],[N+n-k-1]} \right)^{-1}$, we have the following identity
\begin{align} \label{eq-3.31-identity H H^-1}
	\left(
	\begin{matrix}
	zI_N & Y^{(1)} \\
	Y^{(2)} & B
	\end{matrix}
	\right) \left(
	\begin{matrix}
	D & E \\
	F & G
	\end{matrix}
	\right) = \left(
	\begin{matrix}
	D & E \\
	F & G
	\end{matrix}
	\right) \left(
	\begin{matrix}
	zI_N & Y^{(1)} \\
	Y^{(2)} & B
	\end{matrix}
	\right) = I_{N+n-k-1}.
\end{align}

We first control $\|F\|$ with high probability. For any $u \in \bS^{N-1}$, denote $v = Du$ and $w = Fu$ then we have the identity
\begin{align} \label{eq-3.38}
	\left(
	\begin{matrix}
	u \\ 0
	\end{matrix}
	\right) = \left(
	\begin{matrix}
	zI_N & Y^{(1)} \\
	Y^{(2)} & B
	\end{matrix}
	\right) \left(
	\begin{matrix}
	v \\ w
	\end{matrix}
	\right).
\end{align}
Hence, on the event $\left\{ \|X\| \le C_0 \right\}$, we have
\begin{align} \label{eq-3.41-norm of w except kth}
	\left\| B^*Bw \right\| = \left\| B^*Y^{(2)} v \right\| \le C_0 \|v\|.
\end{align}
Denote the event $\cE_{X_n} = \left\{ \|X_n\| \ge c \right\}$ for a small constant $c$. Then by Lemma \ref{Lemma-concentration to space},
\begin{align} \label{eq-3.39-prob on cE Xn}
	\bP \left( \cE_{X_n}^\complement \right) \le \exp (-cN).
\end{align}
As explained in the proof of Lemma \ref{Lemma-similar Coro 14}, by \eqref{eq-3.22-decomposition}, \eqref{eq-3.41-norm of w except kth} and \eqref{eq-3.38}, on the event $\cE_{X_n}$, we have
\begin{align*}
	\|Fu\|^2 = \|w\|^2
	=& \left\| B^*Bw \right\|^2 + |w_k|^2 \\
	\le& C_0^2 \|v\|^2 + \dfrac{|w_k|^2}{c^2} \left\| X_n \right\|^2 \\
	\le& C_0^2 \|v\|^2 + \dfrac{1}{c^2} \left\| Y^{(1)}w - Y^{(1)} B^*Bw \right\|^2 \\
	\le& C_0^2 \|v\|^2 + \dfrac{1}{c^2} \left\| u - zv - Y^{(1)} B^*Bw \right\|^2 \\
	\le& C \left( 1 + \|v\|^2 \right) \\
	=& C \left( 1 + \|Du\|^2 \right).
\end{align*}
Take the supremum over $u$, we obtain $\|F\| \le C(z) (1 + \|D\|)$ on the event $\cE_{X_n}$.

Next, on the event $\left\{ \|X\| \le C_0 \right\}$, we have $\|D\| \ge c$ for a small constant $c$. Suppose not, then by \eqref{eq-3.31-identity H H^-1}, $\|BF\| = \left\| Y^{(2)} D \right\| \le C_0c$. An argument that similar to \eqref{eq-3.22-decomposition}, we may obtain from \eqref{eq-3.31-identity H H^-1} that $I_N = zD + Y^{(1)} F = zD + Y^{(1)} B^*BF + X_n F_{k,[N]}$. Thus, $\left\| I_N - X_n F_{k,[N]} \right\| \le |z|c + C_0^2c$. Then we can find an unit eigenvector $v'$ of the rank $1$ matrix $X_n F_{k,[N]}$ associate to the eigenvalue $0$. Then $\left\|I_N - X_n F_{k,[N]} \right\| \ge \left\| \left(I_N - X_n F_{k,[N]} \right) v' \right\| = \|v'\| = 1$, which leads to a contradiction if we choose $c$ to be small enough. Thus, on the event $\cE_{X_n} \cap \left\{ \|X\| \le C_0 \right\}$, we have $\|F\| \le C \|D\|$.

The control on $\|E\|$ and $\|G\|$ are similar, which are sketched below. By the identity
\begin{align*}
	(u^\intercal,0) = (u^\intercal D, u^\intercal E) \left(
	\begin{matrix}
		zI_N & Y^{(1)} \\
		Y^{(2)} & B
	\end{matrix}
	\right), u \in \bS^{N-1},
\end{align*}
and \eqref{eq-3.31-identity H H^-1}, we have
\begin{align} \label{eq-E norm}
	\|u^\intercal E\|^2
	=& \|u^\intercal E BB^*\|^2 + |(u^\intercal E)_{n-2k}|^2 \nonumber \\
	\le& \|u^\intercal E BB^*\|^2 + \dfrac{1}{c^2} |(u^\intercal E)_{n-2k}|^2 \|X_{n-2k}\|^2 \nonumber \\
	=& \|u^\intercal E BB^*\|^2 + \dfrac{1}{c^2} \|u^\intercal E (I _{n-k-1} - BB^*) Y^{(2)}\|^2 \nonumber \\
	=& \|u^\intercal E BB^*\|^2 + \dfrac{1}{c^2} \|u^\intercal E Y^{(2)} - u^\intercal E BB^* Y^{(2)}\|^2 \nonumber \\
	=& \|u^\intercal D Y^{(1)} B^*\|^2 + \dfrac{1}{c^2} \| u^\intercal - z u^\intercal D + u^\intercal DY^{(1)} B^* Y^{(2)} \|^2 \nonumber \\
	\le& C (1+\|D\|^2)
	\le C\|D\|^2.
\end{align}
on the event $\left\{ \|X\| \le C_0 \right\} \cap \cE_{X_{n-2k}}$, where $\cE_{X_{n-2k}} = \left\{ \|X_{n-2k}\| \ge c \right\}$ for a small constant $c$. Note that by Lemma \ref{Lemma-concentration to space}, we have also
\begin{align} \label{eq-3.36}
	\bP \left( \cE_{X_{n-2k}}^\complement \right) \le \exp (-cN).
\end{align}
By the identity
\begin{align*}
	\left(
	\begin{matrix}
		0 \\ u
	\end{matrix}
	\right) = \left(
	\begin{matrix}
		zI_N & Y^{(1)} \\
		Y^{(2)} & B
	\end{matrix}
	\right) \left(
	\begin{matrix}
		Eu \\ Gu
	\end{matrix}
	\right), u \in \bS^{n-k-1},
\end{align*}
and \eqref{eq-3.31-identity H H^-1}, we have
\begin{align} \label{eq-G norm}
	\|Gu\|^2
	=& \left\| B^*B Gu \right\|^2 + |(Gu)_k|^2 \nonumber \\
	\le& \left\| B^*B Gu \right\|^2 + \dfrac{|(Gu)_k|^2}{c^2} \left\| X_n \right\|^2 \nonumber \\
	\le& \left\| B^*B Gu \right\|^2 + \dfrac{1}{c^2} \left\| Y^{(1)} Gu - Y^{(1)} B^*B Gu \right\|^2 \nonumber \\
	\le& \left\| B^* (I_{n-k-1} - Y^{(2)} E) u \right\|^2 + \dfrac{1}{c^2} \left\| -z Eu - Y^{(1)} B^*B (I_{n-k-1} - Y^{(2)} E) u \right\|^2 \nonumber \\
	\le& C (1 + \|E\|^2)
	\le C (\|D\|^2 + \|E\|^2),
\end{align}
on the event $\left\{ \|X\| \le C_0 \right\} \cap \cE_{X_n}$. Take supremum over $u$ in \eqref{eq-E norm} to obtain $\|E\| \le C \|D\|$ on the event $\left\{ \|X\| \le C_0 \right\} \cap \cE_{X_{n-2k}}$. Then take supremum over $u$ in \eqref{eq-G norm} to obtain $\|G\| \le C \|D\|$ on the event $\left\{ \|X\| \le C_0 \right\} \cap \cE_{X_{n-2k}} \cap \cE_{X_n}$.

Hence, on the event $\cE_{X_{n-2k}} \cap \cE_{X_n} \cap \{\|X\| \le C_0\}$, we have
\begin{align} \label{eq-3.38'}
	\Den^2 =& 1 + \left\| \left( X_{n-k}^*, \left( e_k^{(n-k-1)} \right)^\intercal \right) \left(
	\begin{matrix}
	D & E \\
	F & G
	\end{matrix}
	\right) \right\|^2 \nonumber \\
	\le& 1 + (1+C_0^2) \left( \|D\|^2 + \|E\|^2 + \|F\|^2 + \|G\|^2 \right) \nonumber \\
	\le& C \|D\|^2 \nonumber \\
	\le& C \|D\|_{HS}^2,
\end{align}
where we use the norm relationship $\|D\| \le \|D\|_{HS}$.

Therefore, by \eqref{eq-3.10-distance}, \eqref{eq-3.38'}, \eqref{eq-3.39-prob on cE Xn}, \eqref{eq-3.36} and \eqref{eq-3.30-Prob on Num},
\begin{align*}
	& \bP \left( \dist(H_{[N+n-k],N+n-k}, H_{-(N+n-k)}) \le s, \|X\| \le C_0) \right) \\
	=& \bP \left( \dfrac{\Num}{\Den} \le s, \|X\| \le C_0 \right) \\
	\le& \bP \left( \dfrac{\Num}{\Den} \le s, \cE_{X_{n-2k}} \cap \cE_{X_n} \cap \{\|X\| \le C_0\} \right)
	+ \bP \left( \cE_{X_n}^\complement \cup \cE_{X_{n-2k}}^\complement, \|X\| \le C_0 \right) \\
	\le& \bP \left( \Num \le s C \left\| D \right\|_{HS}, \cE_{X_n}, \|X\| \le C_0 \right) + \exp(-cN) \\
	\le& C(\theta_0, r_H) \sqrt{n^{12/11}s + n^{-1/11}} + \exp(-cN).
\end{align*}
The proof of \eqref{eq-3.9-distance prob} for the case $l = N+n-k$ is finished when we choose $s = n^{-13/11}$. For the case $N+k+1 \le l < N+n-k$ the proof are similar.

\bigskip
The following Lemma has been used previously.
\begin{lemma} \label{Lemma-similar Coro 14}
For any deterministic vector $d \in \bC^N \setminus \{0\}$,
\begin{align*}
	\bP \left( \dfrac{D d}{\left\| D d \right\|} \in \Comp(\theta_0,r_H), \left\| \widetilde{X} \right\| \le C_0 \right)
	\le \exp (-cN),
\end{align*}
\end{lemma}

\begin{proof}{(of Lemma \ref{Lemma-similar Coro 14})}
We first show that for deterministic vector $v \in \bS^{N-1}$ supported on a deterministic index set $I$ with $|I| = \theta_0 N$, for small $t$,
\begin{align} \label{eq-3.21-similar to proposition 12}
	\bP \left( \inf_{w \in \bC^{n-k-1}} \dist \left( H_{[N+n-k-1],[N+n-k-1]} \left(
	\begin{matrix}
	v \\ w
	\end{matrix}
	\right), \Span \left\{ \left(
	\begin{matrix}
	d \\ 0
	\end{matrix}
	\right) \right\} \right) \le t, \left\| \widetilde{X} \right\| \le C_0 \right)
	\le \exp (-cN).
\end{align}
Recalled the definition of $B$, we can see that $B^*B$ is a diagonal matrix, whose $k$-th diagonal entry is zero and other diagonal entries are $1$. Thus,
\begin{align} \label{eq-3.22-decomposition}
	Y^{(1)} w
	=& \sum_{i=1}^{n-k-1} \left( Y^{(1)} \right)_{[N], i} w_i \nonumber \\
	=& \left( Y^{(1)} \right)_{[N], k} w_k + \sum_{i \in [n-k-1] \setminus \{k\}} \left( Y^{(1)} \right)_{[N], i} w_i \nonumber \\
	=& X_n w_k + Y^{(1)} B^*B w.
\end{align}
Recalled the blocking of $H_{[N+n-k-1],[N+n-k-1]}$, we have
\begin{align} \label{eq-3.23-linear equation}
	& \bP \left( \inf_{w \in \bC^{n-k-1}} \dist \left( H_{[N+n-k-1],[N+n-k-1]} \left(
	\begin{matrix}
	v \\ w
	\end{matrix}
	\right), \Span \left\{ \left(
	\begin{matrix}
	d \\ 0
	\end{matrix}
	\right) \right\} \right) \le t, \left\| \widetilde{X} \right\| \le C_0 \right) \nonumber \\
	\le& \bP \left( \exists w \in \bC^{n-k-1}, \exists a \in \bC: \left\| z v + Y^{(1)} w + ad \right\| \le t, \left\| Y^{(2)} v + Bw \right\| \le t, \left\| \widetilde{X} \right\| \le C_0 \right) \nonumber \\
	=& \bP \left( \exists w \in \bC^{n-k-1}, \exists a \in \bC: \left\| z v + X_n w_k + Y^{(1)} B^*B w + ad \right\| \le t, \left\| Y^{(2)} v + Bw \right\| \le t, \left\| \widetilde{X} \right\| \le C_0 \right) \nonumber \\
	\le& \bP \left( \exists w_k, a \in \bC: \left\| z v + X_n w_k - Y^{(1)} B^* Y^{(2)}v + ad \right\| \le t (1+C_0), \left\| \widetilde{X} \right\| \le C_0 \right).
\end{align}

For a large enough constant $C'(t)$ that depends on $t$, by Lemma \ref{Lemma-concentration to space}, we have
\begin{align} \label{eq-3.24-wk large}
	& \bP \left( \exists w_k, a \in \bC: \left\| z v + X_n w_k - Y^{(1)} B^* Y^{(2)}v + ad \right\| \le t (1+C_0), |w_k| \ge C'(t), \left\| \widetilde{X} \right\| \le C_0 \right) \nonumber \\
	\le& \bP \left( \exists w_k, a \in \bC: \left\| z v + X_n w_k + ad \right\| \le C_0^2 + t (1+C_0), |w_k| \ge C'(t), \left\| \widetilde{X} \right\| \le C_0 \right) \nonumber \\
	\le& \bP \left( \exists w_k, a \in \bC: \left\| \dfrac{z}{w_k} v + X_n + \dfrac{a}{w_k} d \right\| \le \dfrac{C_0^2 + t (1+C_0)}{C'(t)}, |w_k| \ge C'(t), \left\| \widetilde{X} \right\| \le C_0 \right) \nonumber \\
	\le& \bP \left( \dist \left( X_n, \Span \left\{ v,d \right\} \right) \le \dfrac{C_0^2 + t (1+C_0)}{C'(t)} \right) \nonumber \\
	\le& \exp (-cN).
\end{align}
We divide the interval $[-C'(t), C'(t)]$ by $-C'(t) = a_1 < \ldots < a_{M(t)} = C'(t)$, such that $|a_{i+1} - a_i| \le 3C'(t)/M(t)$, where $M(t)$ is a large enough constant. Then we have
\begin{align} \label{eq-3.25-wk mall}
	& \bP \left( \exists w_k, a \in \bC: \left\| z v + X_n w_k - Y^{(1)} B^* Y^{(2)}v + ad \right\| \le t (1+C_0), |w_k| < C'(t), \left\| \widetilde{X} \right\| \le C_0 \right) \nonumber \\
	\le& \sum_{i=1}^{M(t)-1} \bP \left( \exists w_k, a \in \bC: \left\| z v + X_n w_k - Y^{(1)} B^* Y^{(2)}v + ad \right\| \le t (1+C_0), w_k \in [a_i,a_{i+1}], \left\| \widetilde{X} \right\| \le C_0 \right) \nonumber \\
	\le& \sum_{i=1}^{M(t)-1} \bP \left( \exists a \in \bC: \left\| z v + X_n a_i - Y^{(1)} B^* Y^{(2)}v + ad \right\| \le t (1+C_0) + \dfrac{3C_0C'(t)}{M(t)}, \left\| \widetilde{X} \right\| \le C_0 \right) \nonumber \\
	\le& \sum_{i=1}^{M(t)-1} \bP \left( \exists a \in \bC: \left\| z v + Y^{(1)} \left( a_i e_k -B^* Y^{(2)}v \right) + ad \right\| \le t (1+C_0) + \dfrac{3C_0C'(t)}{M(t)} \right).
\end{align}
We denote $u_i = a_i e_k -B^* Y^{(2)}v$ then we have
\begin{align*}
	u_i =& a_i e_k - \left(
	\begin{matrix}
	0_{(k-1) \times (n-2k-1)} & 0_{(k-1) \times 1} & I_{k-1} \\
	0_{1 \times (n-2k-1)} & 0 & 0_{1 \times (k-1)} \\
	I_{n-2k-1} & 0_{(n-2k-1)\times 1} & 0_{(n-2k-1) \times (k-1)}
	\end{matrix}
	\right) \left(
	\begin{matrix}
	X_1^* \\
	\vdots \\
	X_{n-k-1}^*
	\end{matrix}
	\right) v \\
	=& \left(
	\begin{matrix}
	-X_{n-2k+1}^*v \\
	\vdots \\
	-X_{n-k-1}^*v \\
	a_i \\
	-X_1^*v \\
	\vdots \\
	-X_{n-2k-1}^*v \\
	\end{matrix}
	\right)
\end{align*}
Note that $v$ is supported on $I$, we can see that $u_i$ is independent of $X_{I^\complement, [n]}$, which implies that $u_i$ is independent of $Y^{(1)}_{I^\complement, [n-k-1]}$. Moreover, by Lemma \ref{Lemma-concentration to space}, we have
\begin{align} \label{eq-3.26-u norm}
	\bP \left( \|u_i\|^2 < c \right)
	\le& \bP \left( \sum_{i \in [n-k-1] \setminus \{n-2k\}} \left( X_i^* v \right)^2 < c \right) \nonumber \\
	=& \bP \left( \left\| \left( X_1, \ldots, X_{n-2k-1}, X_{n-2k+1}, \ldots, X_{n-k-1} \right)^* v \right\|^2 < c \right) \nonumber \\
	\le& \exp \left( -cN \right).
\end{align}
Thus, by \eqref{eq-3.25-wk mall}, \eqref{eq-3.26-u norm} and Lemma \ref{Lemma-concentration to space},
\begin{align} \label{eq-3.27-wk small}
	& \bP \left( \exists w_k, a \in \bC: \left\| z v + X_n w_k - Y^{(1)} B^* Y^{(2)}v + ad \right\| \le t (1+C_0), |w_k| < C'(t), \left\| \widetilde{X} \right\| \le C_0 \right) \nonumber \\
	\le& \sum_{i=1}^{M(t)-1} \bP \left( \exists a \in \bC: \left\| z v + Y^{(1)} u_i + ad \right\| \le t (1+C_0) + \dfrac{3C_0C'(t)}{M(t)} \right) \nonumber \\
	\le& \sum_{i=1}^{M(t)-1} \bP \left( \exists a \in \bC: \left\| z v_{I^\complement} + \left( Y^{(1)} \right)_{I^\complement, [n-k-1]} u_i + ad_{I^\complement} \right\| \le t (1+C_0) + \dfrac{3C_0C'(t)}{M(t)} \right) \nonumber \\
	\le& \sum_{i=1}^{M(t)-1} \bP \left( \exists a \in \bC: \left\| z v_{I^\complement} + \left( Y^{(1)} \right)_{I^\complement, [n-k-1]} u_i + ad_{I^\complement} \right\| \le t (1+C_0) + \dfrac{3C_0C'(t)}{M(t)}, \left\| u_i \right\| \ge c \right) \nonumber \\
	& + M(t) \exp (-cN) \nonumber \\
	\le& \sum_{i=1}^{M(t)-1} \bP \left( \exists a \in \bC: \left\| \dfrac{z}{\|u_i\|} v_{I^\complement} + \left( Y^{(1)} \right)_{I^\complement, [n-k-1]} \dfrac{u_i}{\|u_i\|} + \dfrac{a}{\|u_i\|} d_{I^\complement} \right\| \le Ct + \dfrac{CC'(t)}{M(t)} \right) \nonumber \\
	& + M(t) \exp (-cN) \nonumber \\
	\le& \sum_{i=1}^{M(t)-1} \bP \left( \dist \left( \left( Y^{(1)} \right)_{I^\complement, [n-k-1]} \dfrac{u_i}{\|u_i\|}, \Span \left\{ v_{I^\complement}, d_{I^\complement} \right\} \right) \le Ct + \dfrac{CC'(t)}{M(t)} \right) \nonumber \\
	& + M(t) \exp (-cN) \nonumber \\
	\le& 2M(t) \exp (-cN),
\end{align}
where we need to choose $t$ small, $M(t)$ large. Then \eqref{eq-3.21-similar to proposition 12} follows from \eqref{eq-3.23-linear equation}, \eqref{eq-3.24-wk large} and \eqref{eq-3.27-wk small}.

Similarly, by \eqref{eq-3.22-decomposition} and Lemma \ref{Lemma-concentration to space}, we have
\begin{align} \label{eq-3.29-v=0}
	& \bP \left( \exists w \in \bC^{n-k-1}, \exists r' \ge 0: H_{[N+n-k-1],[N+n-k-1]} \left(
	\begin{matrix}
	0 \\ w
	\end{matrix}
	\right) = \left(
	\begin{matrix}
	r' d \\ 0
	\end{matrix}
	\right), \left\| \widetilde{X} \right\| \le C_0 \right) \nonumber \\
	=& \bP \left( \exists w \in \bC^{n-k-1}, \exists r' \ge 0: Y^{(1)} w = r' d, Bw = 0, \left\| \widetilde{X} \right\| \le C_0 \right) \nonumber \\
	=& \bP \left( \exists w_k \in \bC, \exists r' \ge 0: X_n w_k = r' d, Bw = 0, \left\| \widetilde{X} \right\| \le C_0 \right) \nonumber \\
	=& \bP \left( \exists w_k \in \bC \setminus \{0\}, \exists r' \ge 0: X_n = \dfrac{r'}{w_k}d \right) \nonumber \\
	=& \bP \left( X_n \in \Span \left\{ d \right\} \right) \nonumber \\
	\le& \exp (-cN).
\end{align}

By Lemma \ref{Lemma-Compressible}, the set of compressible vectors $\Comp(\theta_0, r_H)$ lies in a $r_H$-neighbourhood of $\bS_I^{N-1}$ for some $I \subseteq [N]$ with $|I| = \theta_0 N$, and the set $\bS_I^{N-1}$ has a $r_H$-net of cardinal number bounded by $(3/r_H)^{2 \theta_0 N}$. Moreover, by \eqref{eq-3.9'-||H||} and Lemma \ref{Lemma-singular value interacting}, we have $\left\| H_{[N+n-k-1],[N+n-k-1]} \right\| \le \|H\| \le |z|+1+C_0$. Thus, by \eqref{eq-3.21-similar to proposition 12} and \eqref{eq-3.29-v=0},
\begin{align*}
	& \bP \left( \dfrac{D d}{\left\| D d \right\|} \in \Comp(\theta_0,r_H), \left\| \widetilde{X} \right\| \le C_0 \right) \\
	=& \bP \left( \exists v \in \Comp(\theta_0,r_H), r \ge 0, w \in \bC^{n-k-1}: H_{[N+n-k-1],[N+n-k-1]} \left(
	\begin{matrix}
	rv \\ w
	\end{matrix}
	\right) = \left(
	\begin{matrix}
	d \\ 0
	\end{matrix}
	\right), \left\| \widetilde{X} \right\| \le C_0 \right) \\
	=& \bP \left( \exists w \in \bC^{n-k-1}: H_{[N+n-k-1],[N+n-k-1]} \left(
	\begin{matrix}
	0 \\ w
	\end{matrix}
	\right) = \left(
	\begin{matrix}
	d \\ 0
	\end{matrix}
	\right), \left\| \widetilde{X} \right\| \le C_0 \right) \\
	&+ \bP \left( \exists v \in \Comp(\theta_0,r_H), r>0, w \in \bC^{n-k-1}: H_{[N+n-k-1],[N+n-k-1]} \left(
	\begin{matrix}
	v \\ w
	\end{matrix}
	\right) = \left(
	\begin{matrix}
	d/r \\ 0
	\end{matrix}
	\right), \left\| \widetilde{X} \right\| \le C_0 \right) \\
	\le& \exp (-cN) + \binom{N}{\theta_0 N} \left( \dfrac{3}{r_H} \right)^{2\theta_0 N} \nonumber \\
	& \times \bP \left( \exists v \in B(v_0,2r_H), r>0, w \in \bC^{n-k-1}: H_{[N+n-k-1],[N+n-k-1]} \left(
	\begin{matrix}
	v \\ w
	\end{matrix}
	\right) = \left(
	\begin{matrix}
	d/r \\ 0
	\end{matrix}
	\right), \left\| \widetilde{X} \right\| \le C_0 \right) \\
	\le& \exp (-cN) + \binom{N}{\theta_0 N} \left( \dfrac{3}{r_H} \right)^{2\theta_0 N} \nonumber \\
	&\times \bP \left( \exists w \in \bC^{n-k-1}: \dist \left( H_{[N+n-k-1],[N+n-k-1]} \left(
	\begin{matrix}
	v_0 \\ w
	\end{matrix}
	\right), \Span \left\{ \left(
	\begin{matrix}
	d \\ 0
	\end{matrix}
	\right) \right\} \right) \le 2r_H(|z|+1+C_0), \left\| \widetilde{X} \right\| \le C_0 \right) \\
	\le& \exp (-cN) + \binom{N}{\theta_0 N} \left( \dfrac{3}{r_H} \right)^{2\theta_0 N} \exp (-cN) \\
	\le& \exp (-cN).
\end{align*}
In the last inequality, we just use the Stirling formula as what we do at the end of Step 1. We may replace the $\theta_0$ by a smaller one if necessary. Hence, the proof of Lemma \ref{Lemma-similar Coro 14} is finished.
\end{proof}

\begin{remark} \label{Remark--similar Coro 14}
By a similar argument, one can show that
\begin{align*}
	\bP \left( \dfrac{d^* D}{\left\| d^* D \right\|} \in \Comp(\theta_0,r_H), \left\| \widetilde{X} \right\| \le C_0 \right) \le \exp (-cN).
\end{align*}
\end{remark}

\subsubsection{Case of \texorpdfstring{$N+1 \le l \le N+k$}{N+1 <= l <= N+k}}\label{Case2}
The estimation is similar to the previous case of $N+k+1\le l\le N+n-k$ and the proof is sketched as follows.
Without loss of generality, we only estimate \eqref{eq-3.10-distance} for the case $l = N+1$. First of all, we have
\begin{align*}
	&H_{N+1,N+1} = 0, \
	H_{N+1,[N+n-k] \setminus \{N+1\}} = \left( X_1^*, \left( e_k^{(n-k-1)} \right)^\intercal \right), \nonumber \\
	&H_{[N+n-k] \setminus \{N+1\},N+1} = \left(
	\begin{matrix}
		X_{n-k+1} \\
		e_{n-2k}^{(n-k+1)}
	\end{matrix}
	\right), \nonumber \\
	&H_{[N+n-k] \setminus \{N+1\},[N+n-k] \setminus \{N+1\}} \nonumber \\
	=& \left(
	\begin{matrix}
		zI_N & \left( X_{n-k+2}, \ldots, X_n, X_{k+1}, \ldots, X_{n-k} \right) \\
		\left( X_2, \ldots, X_{n-k} \right)^* & \left( e_{n-2k+1}^{(n-k-1)}, \ldots, e_{n-k-1}^{(n-k-1)}, 0, e_1^{(n-k-1)}, \ldots, e_{n-2k-1}^{(n-k-1)} \right)
	\end{matrix}
	\right).
\end{align*}

By showing that the determinant of $H_{[N+n-k] \setminus \{N+1\}, [N+n-k] \setminus \{N+1\}}$ is a non-zero polynomial of the entries of $X$, one can deduce the invertibility of $H_{[N+n-k] \setminus \{N+1\}, [N+n-k] \setminus \{N+1\}}$. Next, we denote
\begin{align*}
	H_{[N+n-k] \setminus \{N+1\}, [N+n-k] \setminus \{N+1\}}^{-1}
	= \left(
	\begin{matrix}
		D & E \\
		F & G
	\end{matrix}
	\right),
\end{align*}
where $D \in \bC^{N \times N}$ and $G \in \bC^{(n-k-1) \times (n-k-1)}$. Denote a random vector $Y = (X_1^*, X_{n-k+1}^*)^*$, then by \eqref{eq-3.11-Def Num},
\begin{align*}
	\Num =& \left| X_1^* D X_{n-k+1} + X_1^* E e_{n-2k}^{(n-k+1)} + \left( e_k^{(n-k-1)} \right)^\intercal F X_{n-k+1} + \left( e_k^{(n-k-1)} \right)^\intercal G e_{n-2k}^{(n-k+1)} \right| \\
	=& \left| Y^* \left(
	\begin{matrix}
		0_{N \times N} & D \\
		0_{N \times N} & 0_{N \times N}
	\end{matrix}
	\right) Y + Y^* \left(
	\begin{matrix}
		E e_{n-2k}^{(n-k+1)} \\
		0_{N \times 1}
	\end{matrix}
	\right) + \left(0_{1 \times N}, \left( e_k^{(n-k-1)} \right)^\intercal F \right) Y + \left( e_k^{(n-k-1)} \right)^\intercal G e_{n-2k}^{(n-k+1)} \right| \\
	=& \left| Y^* \widetilde{D} Y + Y^* \widetilde{b}^{(1)} + \left(\widetilde{b}^{(2)} \right)^* Y + \widetilde{a} \right|,
\end{align*}
where
\begin{align*}
	& \widetilde{D} = \left(
	\begin{matrix}
		0_{N \times N} & D \\
		0_{N \times N} & 0_{N \times N}
	\end{matrix}
	\right), \
	\widetilde{b}^{(1)} = \left(
	\begin{matrix}
		E e_{n-2k}^{(n-k+1)} \\
		0_{N \times 1}
	\end{matrix}
	\right), \\
	& \left(\widetilde{b}^{(2)} \right)^* = \left(0_{1 \times N}, \left( e_k^{(n-k-1)} \right)^\intercal F \right), \
	\widetilde{a} = \left( e_k^{(n-k-1)} \right)^\intercal G e_{n-2k}^{(n-k+1)}.
\end{align*}
Similar to Step (a) in Section \ref{Case1}, we introduce an independent family $\xi = \{\xi_1, \ldots, \xi_N\}$ of Bernoulli random variables and denote $I = \left\{ i \in [N]: \xi_i = 1 \right\}$, $\widetilde{I} = I \cup (I+N) = I \cup \left\{ i \in [2N] \setminus [N] : i-N \in I \right\}$.
Choose independent random vectors $\widetilde{x}, \widetilde{x}', \widetilde{x}'' \overset{d}{=} Y$ and set $\widetilde{u} = \left( \widetilde{x} \right)_{\widetilde{I}}$, $\widetilde{v} = \left( \widetilde{x}' \right)_{\widetilde{I}^\complement}$ and $\widetilde{w} = \left( \widetilde{x}'' \right)_{\widetilde{I}^\complement}$.
Denote $\widetilde{X} = (X_2, \ldots, X_{n-k}, X_{n-k+2}, \ldots, X_n)$. Then by Cauchy-Schwarz inequality and Lemma \ref{Lemma-quadratic form bound},
\begin{align*}
	& \bP \left( \Num \le t, \|X\| \le C_0 \right)^2 \\
	\le& \bP \left( \left| Y^* \widetilde{D} Y + Y^* \widetilde{b}^{(1)} + \left(\widetilde{b}^{(2)} \right)^* Y + \widetilde{a} \right| \le t, \left\| \widetilde{X} \right\| \le C_0 \right)^2 \\
	=& \left( \bE_{\widetilde{X}} \left[ \bE_{Y} \left[ 1_{\left| Y^* \widetilde{D} Y + Y^* \widetilde{b}^{(1)} + \left(\widetilde{b}^{(2)} \right)^* Y + \widetilde{a} \right| \le t} \right] 1_{\left\| \widetilde{X} \right\| \le C_0} \right] \right)^2 \\
	\le& \bE_{\widetilde{X}} \left[ \left( \bE_{Y} \left[ 1_{\left| Y^* \widetilde{D} Y + Y^* \widetilde{b}^{(1)} + \left(\widetilde{b}^{(2)} \right)^* Y + \widetilde{a} \right| \le t} \right] \right)^2 1_{\left\| \widetilde{X} \right\| \le C_0} \right] \\
	\le& \bE_{\widetilde{X}} \left[ \bE_{\widetilde{v}, \widetilde{w}} \left[ \cS_{\widetilde{u}} \left( (\widetilde{v} - \widetilde{w})^* \widetilde{D}_{\widetilde{I}^\complement, \widetilde{I}} \widetilde{u} + \widetilde{u}^* \widetilde{D}_{\widetilde{I}, \widetilde{I}^\complement} (\widetilde{v} - \widetilde{w}), 2t \right) \right] 1_{\left\| \widetilde{X} \right\| \le C_0} \right] \\
	=& \bE_{\widetilde{X}} \left[ \bE_{\widetilde{v}, \widetilde{w}} \left[ \cS_{\widetilde{u}} \left( (\Pi_{\widetilde{I}^\complement} (\widetilde{x}' - \widetilde{x}''))^* \widetilde{D} \left( \Pi_{\widetilde{I}} \widetilde{x} \right) + (\Pi_{\widetilde{I}} \widetilde{x})^* \widetilde{D} \left( \Pi_{\widetilde{I}^\complement}(\widetilde{x}' - \widetilde{x}'') \right), 2t \right) 1_{\|\widetilde{X}\| \le C_0} \right] \right].
\end{align*}
One can show that Lemma \ref{Lemma-similar Coro 14} still holds for $D$.
We can define $y, \tilde{y}, \alpha, \tilde{\alpha}, W_i$ as we do in Section \ref{Case1}, where $D$, $I$, $x'$ and $x''$ should be replaced by $\widetilde{D}$, $\widetilde{I}$, $\widetilde{x}'$ and $\widetilde{x}''$, respectively.
Let $\widetilde{J}$ be the index set given by \eqref{eq-def-J} with $N$ replaced by $2N$. Then we can obtain the corresponding upper bound \eqref{eq-3.16-upper bound of cS} with $I$, $J$ and $D$ replaced by $\widetilde{I}$, $\widetilde{J}$ and $\widetilde{D}$, respectively.
If $I^\complement \not= \emptyset$, then $\Pi_{I^\complement+N}(\widetilde{x}' - \widetilde{x}'') = 0$ with probability zero since the entries of $Y$ have continuous density. 
Then by the Lemma \ref{Lemma-similar Coro 14}, one can still obtain that the probability of $y \in \Comp(\theta_0, r_H)$ is at most $\exp(-cN)$. Besides,
\begin{align*}
	\bP \left( I^\complement = \emptyset \right) = p^N = \exp(-cN).
\end{align*}
Thus, the probability of $y \in \Incomp(\theta_0,r_H)$ is at least $1 - \exp(-cN)$.
On the event $\{ y \in \Incomp(\theta_0,r_H) \} \cap \cE_I$, we still have $|\widetilde{I} \cap \widetilde{J}| \ge \theta_0 N/3$ for a small constant c that only depends on $\theta_0$. Then the computation of \eqref{eq-3.29-prob cE alpha} is still valid with $D$ replaced by $\widetilde{D}$, the unit vector $e_i^{(N)}$ replaced by $e_i^{(2N)}$ and the range of the index $j$ should be $[2N]$. Thus, one may derive \eqref{eq-3.30-Prob on Num}.

The computation of $\Den$ given by \eqref{eq-3.12-Def Den} as well as the estimation
\begin{align*}
	\bP \left( \dfrac{\Num}{\Den} \le s, \left\| X \right\| \le C_0 \right)
	\le C(\theta_0, r_H) \sqrt{n^{12/11}s + n^{-1/11}} + \exp(-cN)
\end{align*}
can be deduced step by step as in Section \ref{Case1}. By choosing $s = n^{-13/11}$, we obtain \eqref{eq-3.9-distance prob} for the case $l = N+1$.

\subsubsection{Case of \texorpdfstring{$1 \le l \le N$}{1 <= l <= N}}\label{Case3}
The estimation is similar to the previous case of $N+k+1\le l\le N+n-k$ and the proof is sketched as follows.
Without loss of generality, we only estimate \eqref{eq-3.10-distance} for the case $l = N$. First of all, we have
\begin{align*}
	&H_{N,N} = z, \
	H_{N,[N+n-k] \setminus \{N\}} = \left( 0_{1 \times (N-1)}, \left( X_{N,[n] \setminus [n-k]}, X_{N, [n-k] \setminus [k]} \right) \right), \nonumber \\
	& H_{[N+n-k] \setminus \{N\},N} = \left(
	\begin{matrix}
	0_{(N-1) \times 1} \\
	\left( X_{N, [n-k]} \right)^*
	\end{matrix}
	\right), \nonumber \\
	& H_{[N+n-k] \setminus \{N\},[N+n-k] \setminus \{N\}} = \left(
	\begin{matrix}
	zI_{N-1} & \left( X_{[N-1], [n] \setminus [n-k]}, X_{[N-1], [n-k] \setminus [k]} \right) \\
	\left( X_{[N-1],[n-k]} \right)^* & \left( e_{n-2k+1}^{(n-k)}, \ldots, e_{n-k}^{(n-k)}, e_1^{(n-k)}, \ldots, e_{n-2k}^{(n-k)} \right)
	\end{matrix}
	\right).
\end{align*}

By showing that the determinant of $H_{[N+n-k] \setminus \{N\}, [N+n-k] \setminus \{N\}}$ is a non-zero polynomial of the entries of $X$, one can deduce the invertibility of $H_{[N+n-k] \setminus \{N\}, [N+n-k] \setminus \{N\}}$. Next, we denote
\begin{align*}
	H_{[N+n-k] \setminus \{N\}, [N+n-k] \setminus \{N\}}^{-1}
	= \left(
	\begin{matrix}
	D & E \\
	F & G
	\end{matrix}
	\right),
\end{align*}
where $D \in \bC^{(N-1) \times (N-1)}$ and $G \in \bC^{(n-k) \times (n-k)}$. Denote a row random vector
\begin{align*}
	Y = \left( Y_1, Y_2, Y_3 \right) = \left( X_{N,[k]}, X_{N,[n] \setminus [n-k]}, X_{N,[n-k] \setminus [k]} \right),
\end{align*}
then by \eqref{eq-3.11-Def Num},
\begin{align*}
	\Num =& \left| z - \left( 0_{1 \times (N-1)}, Y_2, Y_3 \right) \left(
	\begin{matrix}
	D & E \\
	F & G
	\end{matrix}
	\right) \left(0_{1 \times (N-1)}, Y_1, Y_3 \right)^* \right| \\
	=& \left| z - \left( Y_2, Y_3 \right) G \left( Y_1, Y_3 \right)^* \right| \\
	=& \left| z - Y \left(
	\begin{matrix}
	0_{k \times k} & 0_{k \times k} & 0_{k \times (n-2k)} \\
	G_{[k], [k]} & 0_{k \times k} & G_{[k], [n-2k]} \\
	G_{[n-2k], [k]} & 0_{(n-2k) \times k} & G_{[n-2k], [n-2k]}
	\end{matrix}
	\right) Y^* \right|.
\end{align*}
We denote the matrix above as $\widetilde{G}$ then $\Num = \left| z - Y\widetilde{G} Y^* \right|$. Similar to Step (a) in Section \ref{Case1}, we introduce an independent family $\xi = \{\xi_1, \ldots, \xi_n\}$ of Bernoulli random variables and denote $I = \left\{ i \in [n]: \xi_i = 1 \right\}$. Choose independent random vectors $x,x',x'' \overset{d}{=} Y$ and set $u = (x)_I$, $v = (x')_{I^\complement}$ and $w = (x'')_{I^\complement}$. Denote $\widetilde{X} = X_{[N-1], [n]}$. Then by Cauchy-Schwarz inequality and Lemma \ref{Lemma-quadratic form bound},
\begin{align*}
	& \bP \left( \Num \le t, \|X\| \le C_0 \right)^2 \\
	\le& \bP \left( \left| z - Y\widetilde{G} Y^* \right| \le t, \left\| \widetilde{X} \right\| \le C_0 \right)^2 \\
	=& \left( \bE_{\widetilde{X}} \left[ \bE_{X_{N,[n]}} \left[ 1_{\left| z - Y\widetilde{G} Y^* \right| \le t} \right] 1_{\left\| \widetilde{X} \right\| \le C_0} \right] \right)^2 \\
	\le& \bE_{\widetilde{X}} \left[ \left( \bE_{Y} \left[ 1_{\left| z - Y \widetilde{G} Y^* \right| \le t} \right] \right)^2 1_{\left\| \widetilde{X} \right\| \le C_0} \right] \\
	\le& \bE_{\widetilde{X}} \left[ \bE_{v,w} \left[ \cS_u \left( (v-w)^* \widetilde{G}_{I^\complement, I} u + u^* \widetilde{G}_{I, I^\complement} (v-w), 2t \right) \right] 1_{\left\| \widetilde{X} \right\| \le C_0} \right] \\
	=& \bE_{\widetilde{X}} \left[ \bE_{v,w} \left[ \cS_u \left( (\Pi_{I^\complement}(x'-x''))^* \widetilde{G} \left( \Pi_Ix \right) + (\Pi_Ix)^* \widetilde{G} \left( \Pi_{I^\complement}(x'-x'') \right), 2t \right) 1_{\|\widetilde{X}\| \le C_0} \right] \right].
\end{align*}

Then we can define $y, \tilde{y}, \alpha, \tilde{\alpha}, W_i$ as we do in Section \ref{Case1}, where $D$ should be replaced by $\widetilde{G}$.
One can compute the corresponding upper bound \eqref{eq-3.16-upper bound of cS}, where $J$ should be given in \eqref{eq-def-J} with $N$ replaced by $n$.
Besides, one can also show that Lemma \ref{Lemma-similar Coro 14} holds with $D$ replaced by $G$. For $d = (d_1^*, d_2^*, d_3^*)^* \in \bC^n$, note that
\begin{align*}
	\widetilde{G}d = \left(
	\begin{matrix}
	0_{k \times 1} \\
	G \left(
	\begin{matrix}
	d_1 \\ d_3
	\end{matrix}
	\right)
	\end{matrix}
	\right),
\end{align*}
we can see that $\widetilde{G} d \in \Comp(\theta_0, r_H)$ if and only if $G \left( d_1^*, d_3^* \right)^* \in \Comp(\theta_0, r_H)$.
Thus, Lemma \ref{Lemma-similar Coro 14} holds with $D$ replaced by $\widetilde{G}$, if $\Pi_{[k] \cup ([n] \setminus [2k])}d \not= 0$.

On the event $\left\{ I^\complement \nsubseteq [2k] \setminus [k] \right\}$, $\Pi_{[k] \cup ([n] \setminus [2k])} (x'-x'') = 0$ with probability zero since the entries of $Y$ have continuous density. Then by the Lemma \ref{Lemma-similar Coro 14}, one can still obtain that the probability of $y \in \Comp(\theta_0, r_H)$ is at most $\exp(-cN)$. Besides,
\begin{align*}
	\bP \left( I^\complement \subseteq [2k] \setminus [k] \right)
	= \bP \left( \xi_i = 1, \forall i \in [k] \cup ([n] \setminus [2k]) \right)
	= p^{n-k}
	= \exp(-cN).
\end{align*}
Thus, the probability of $y \in \Incomp(\theta_0,r_H)$ is at least $1 - \exp(-cN)$.
Similarly, one can show that the probability of the event $\cE_I' = \{|I| > n(1-\theta_0/3)\}$ is at least $1 - \exp(-cN)$.
On the event $\{y \in \Incomp(\theta_0,r_H)\} \cap \cE_I'$, we still have $|I \cap J| \ge \theta_0 n/6$.

Note that
\begin{align*}
	\left( e_i^{(n)} \right)^* \widetilde{G} =
	\begin{cases}
	0, & i \in [2k] \setminus [k], \\
	\left( e_i^{(n-k)} \right)^* G, & i \in [k] \cup ([n] \setminus [2k]).
	\end{cases}
\end{align*}
So
\begin{align*}
	\left\| \widetilde{G} \left( \Pi_{I^\complement}(x'-x'') \right) \right\|^2
	= \sum_{i \in [k] \cup ([n] \setminus [2k])} \left( \left( e_i^{(n)} \right)^* \widetilde{G} \left( \Pi_{I^\complement}(x'-x'') \right) \right)^2.
\end{align*}
Then the computation of \eqref{eq-3.29-prob cE alpha} is still valid with $D$ replaced by $\widetilde{G}$ and the index $i$ should be in the index set $[k] \cup ([n] \setminus [2k])$. Thus, one may derive \eqref{eq-3.30-Prob on Num} with $D$ replaced by $G$.

Next, we compute $\Den$ given by \eqref{eq-3.12-Def Den}. Recalled the definition of $H$ in \eqref{eq-3.13-entries of H} and the blocking of $\left( H_{[N+n-k] \setminus [N],[N+n-k] \setminus [N]} \right)^{-1}$, we have the following identity
\begin{align*}
	I_{N+n-k-1} =& \left(
	\begin{matrix}
	zI_{N-1} & H_{[N-1], [N+n-k] \setminus [N]} \\
	H_{[N+n-k] \setminus [N], [N-1]} & H_{[N+n-k] \setminus [N], [N+n-k] \setminus [N]}
	\end{matrix}
	\right) \left(
	\begin{matrix}
	D & E \\
	F & G
	\end{matrix}
	\right) \\
	=& \left(
	\begin{matrix}
	D & E \\
	F & G
	\end{matrix}
	\right) \left(
	\begin{matrix}
	zI_{N-1} & H_{[N-1], [N+n-k] \setminus [N]} \\
	H_{[N+n-k] \setminus [N], [N-1]} & H_{[N+n-k] \setminus [N], [N+n-k] \setminus [N]}
	\end{matrix}
	\right).
\end{align*}
Note that the submatrix $H_{[N+n-k] \setminus [N], [N+n-k] \setminus [N]}$ is a permutation matrix, $H_{[N-1], [N+n-k] \setminus [N]}$ is a submatrix of $X$ up to a permutation, $H_{[N+n-k] \setminus [N], [N-1]}$ is a submatrix of $X^*$. Thus, on the event $\{ \|X\| \le C_0 \}$, we have
\begin{align*}
	\left\| F \right\|
	= \dfrac{1}{|z|} \left\| F zI_{N-1} \right\|
	= \dfrac{1}{|z|} \left\| G H_{[N+n-k] \setminus [N], [N-1]} \right\|
	\le \dfrac{C_0}{|z|}  \|G\|.
\end{align*}
Moreover, on the event $\{ \|X\| \le C_0 \}$, suppose that $\|G\| \le c'$ for a small constant $c'$ that may depend on $|z|$ and will be determined later, then $\|F\| \le C_0c'/|z|$. Thus,
\begin{align*}
	1 = \left\| I_{n-k} \right\| 
	= \left\| F H_{[N-1], [N+n-k] \setminus [N]} + G H_{[N+n-k] \setminus [N], [N+n-k] \setminus [N]} \right\|
	\le C_0 \|F\| + \|G\|
	\le \dfrac{C_0^2c'}{|z|} + c'.
\end{align*}
We may choose $c' < |z|/\left( |z| + C_0^2 \right)$ to reach a contradiction. Thus, $\|G\| > c'$ on the event $\{ \|X\| \le C_0 \}$. Therefore, on the event $\{ \|X\| \le C_0 \}$,
\begin{align*}
	\Den^2 =& 1 + \left\| \left( X_{N,[n] \setminus [n-k]}, X_{N, [n-k] \setminus [k]} \right) F \right\|^2 + \left\| \left( X_{N,[n] \setminus [n-k]}, X_{N, [n-k] \setminus [k]} \right) G \right\| \\
	\le& 1 + C_0^2 \|F\|^2 + C_0^2 \|G\|^2 \\
	\le& C\|G\|^2.
\end{align*}
Thus, \eqref{eq-3.9-distance prob} follows from the estimations on $\Num$ and $\Den$.

\subsection{Estimate \texorpdfstring{\eqref{eq-3.5-Incompressible}}{(3.5)} for imcompressible vectors for the case \texorpdfstring{$2k+1 > n$}{2k+1 > n}}\label{step3}

We now establish estimation \eqref{eq-3.5-Incompressible} for the case $2k+1 > n$ with $\theta = \theta_0$ and $\rho = r_H$. The proof is similar to that in Section \ref{step2} and is sketched as follows.

It is enough to prove \eqref{eq-3.9-distance prob}. Moreover, \eqref{eq-3.10-distance}, \eqref{eq-3.11-Def Num} and \eqref{eq-3.12-Def Den} are still valid.

\subsubsection{Case of \texorpdfstring{$N+1 \le l \le N+n-k$}{N+1 <= l <= N+n-k}}\label{Case4}

Without loss of generality, we only estimate \eqref{eq-3.10-distance} for the case $l = N+n-k$. Recalled the definition of $H$, we have
\begin{align*}
	&H_{N+n-k,N+n-k} = 1, \
	H_{N+n-k,[N+n-k-1]} = \left( X_{n-k}^*, 0_{1 \times (n-k-1)} \right), \\
	& H_{[N+n-k-1],N+n-k} = \left(
	\begin{matrix}
		X_n \\
		0_{(n-k-1) \times 1} \\
	\end{matrix}
	\right), \nonumber \\
	& H_{[N+n-k-1],[N+n-k-1]} = \left(
	\begin{matrix}
		zI_N & \left( X_{k+1}, \ldots, X_{n-1} \right) \\
		\left( X_1, \ldots, X_{n-k-1} \right)^* & I_{n-k-1}
	\end{matrix}
	\right).
\end{align*}
One can follow the argument in Section \ref{Case2} to obtain the existence of
\begin{align*}
	H_{[N+n-k-1],[N+n-k-1]}^{-1} = \left(
	\begin{matrix}
		D & E\\
		F & G
	\end{matrix}
	\right),
\end{align*}
and the estimation \eqref{eq-3.30-Prob on Num} of $\Num$. For the estimation of $\Den$, one can follow the argument in Section \ref{Case3} to obtain
\begin{align*}
	\Den^2 \le C \|D\|^2,
\end{align*}
on the event $\{\|X\| \le C_0\}$. The estimation \eqref{eq-3.9-distance prob} follows from the estimations on $\Num$ and $\Den$.

\subsubsection{Case of \texorpdfstring{$1 \le l \le N$}{1 <= l <= N}}\label{Case5}

Without loss of generality, we only estimate \eqref{eq-3.10-distance} for the case $l = N$. Recalled the definition of $H$, we have
\begin{align*}
	&H_{N,N} = z, \
	H_{N,[N+n-k] \setminus \{N\}} = \left( 0_{1 \times (N-1)}, X_{N,[n] \setminus [k]} \right), \
	H_{[N+n-k] \setminus \{N\},N} = \left(
	\begin{matrix}
		0_{(N-1) \times 1} \\
		\left( X_{N, [n-k]} \right)^*
	\end{matrix}
	\right), \nonumber \\
	& H_{[N+n-k] \setminus \{N\},[N+n-k] \setminus \{N\}} = \left(
	\begin{matrix}
		zI_{N-1} & X_{[N-1], [n] \setminus [k]} \\
		\left( X_{[N-1],[n-k]} \right)^* & I_{n-k}
	\end{matrix}
	\right).
\end{align*}
One can follow the argument in Section \ref{Case2} to obtain the existence of
\begin{align*}
	H_{[N+n-k-1],[N+n-k-1]}^{-1} = \left(
	\begin{matrix}
		D & E\\
		F & G
	\end{matrix}
	\right),
\end{align*}
and the estimation \eqref{eq-3.30-Prob on Num} of $\Num$ with $D$ replaced by $G$. For the estimation of $\Den$, one can follow the argument in Section \ref{Case3} to obtain
\begin{align*}
	\Den^2 \le C \|G\|^2,
\end{align*}
on the event $\{\|X\| \le C_0\}$. The estimation \eqref{eq-3.9-distance prob} follows from the estimations on $\Num$ and $\Den$.

\section{Limiting eigenvalue empirical distribution}
\label{sec-LSD}

Though the small rank perturbation for Hermitian matrices fails in general, with the estimation on least singular value in Section \ref{sec-least singular value}, it turns out that the limit of $\mu_{\Yn}$ when $k/n = o(\ln^{-1} n)$ is the same as the limit of $\mu_{\Zn}$. The detail argument is developed in Section \ref{sec-LSD-gamma1=0}. In addition, we establish the limit of $\mu_{\Yn}$ when $k \ge n/2$ in Section \ref{sec-LSD-general}.

%We derive the limit of $\mu_{\Yn}$ for the case that $k/n = o(\ln^{-1} n)$ in Section \ref{sec-LSD-gamma1=0} and for the case $k \ge n/2$ in Section \ref{sec-LSD-general}.

\subsection{The case \texorpdfstring{$k/n = o(\ln^{-1} n)$}{k=o(n/ln n)}} \label{sec-LSD-gamma1=0}

Let
\begin{align*}
	g(x) = \dfrac{x(1-\gamma_0 + 2x)^2}{1+x}, \quad x \in [0 \vee (\gamma_0 - 1), \gamma_0].
\end{align*}
then \cite{Bose2019} showed that $g$ is increasing and invertible on its domain. Moreover, under the conditions (C1) and (C2), \cite{Bose2019} also showed that $\mu_{\Zn}$ of the matrix in \eqref{def-Z} converges to a deterministic rotation invariant probability measure $\mu^{(\gamma_0)}$ in probability. Moreover, the distribution function of the radial component of $\mu^{(\gamma_0)}$ is
\begin{align*}
	\mu^{(\gamma_0)} \left( B(0,r) \right) =
	\begin{cases}
	\gamma_0^{-1} g^{-1}(r^2), & 0 \le r \le \gamma_0^{1/2}(\gamma_0+1)^{1/2}, \\
	1, & r > \gamma_0^{1/2}(\gamma_0+1)^{1/2},
	\end{cases}
\end{align*}
if $\gamma_0 \le 1$, and
\begin{align*}
	\mu^{(\gamma_0)} \left( B(0,r) \right) =
	\begin{cases}
	1 - \gamma_0^{-1}, & 0 \le r \le (\gamma_0 - 1)^{3/2} \gamma_0^{-1/2}, \\
	\gamma_0^{-1} g^{-1}(r^2), & (\gamma_0 - 1)^{3/2} \gamma_0^{-1/2} \le r \le \gamma_0^{1/2}(\gamma_0+1)^{1/2}, \\
	1, & r > \gamma_0^{1/2}(\gamma_0+1)^{1/2},
	\end{cases}
\end{align*}
if $\gamma_0 > 1$. Recall the matrix of interest $\Yn$ in \eqref{def-Y}.

\begin{theorem} \label{Thm-LSD-k=1}
Let $k = 1$ and let $N$ satisfy \eqref{eq-1.0-N/n}. Assume that the conditions (C1) and(C2) hold. Then $\mu_{\Yn }$ converges weakly to $\mu^{(\gamma_0)}$ in probability.
\end{theorem}

\begin{proof}
Note that $A^{(n)} - J^{(n)}$ is a rank one matrix, so is $\Yn  - \Zn$. Thus, by Lemma \ref{Lemma-singular value sum}, we have
\begin{align*}
	s_i \left( \Yn  - zI_N \right) \ge s_{i+1} \left( \Zn - zI_N \right), \ \forall i \in [N-1].
\end{align*}
Thus, for small $\delta \in (0,1)$, on the event $\left\{ s_N \left( \Yn  - zI_N \right) < \delta \right\}$,
\begin{align*}
	\left| \int_0^{\delta} \ln(\lambda) d \nu_{\left( \Yn  - zI_N \right)} (\lambda) \right|
	=& \dfrac{1}{N} \sum_{s_i \left( \Yn  - zI_N \right) < \delta} \left| \ln \left( s_i \left( \Yn  - zI_N \right) \right) \right| \\
	=& \dfrac{1}{N} \left| \ln \left( s_N \left( \Yn  - zI_N \right) \right) \right| \\
	& + \dfrac{1}{N} \sum_{s_i \left( \Yn  - zI_N \right) < \delta, i<N} \left| \ln \left( s_i \left( \Yn  - zI_N \right) \right) \right| \\
	\le& \dfrac{1}{N} \left| \ln \left( s_N \left( \Yn  - zI_N \right) \right) \right| \\
	& + \dfrac{1}{N} \sum_{s_i \left( \Yn  - zI_N \right) < \delta, i<N} \left| \ln \left( s_{i+1} \left( \Zn - zI_N \right) \right) \right| \\
	\le& \dfrac{1}{N} \left| \ln \left( s_N \left( \Yn  - zI_N \right) \right) \right| \\
	& + \dfrac{1}{N} \sum_{s_i \left( \Zn - zI_N \right) < \delta} \left| \ln \left( s_i \left( \Zn - zI_N \right) \right) \right| \\
	=& \dfrac{1}{N} \left| \ln \left( s_N \left( \Yn  - zI_N \right) \right) \right| + \left| \int_0^{\delta} \ln(\lambda) d \nu_{\left( \Zn - zI_N \right)} (\lambda) \right|.
\end{align*}
Thus, for any $\epsilon > 0$, by Theorem \ref{Thm-least singular value}, we have
\begin{align*}
	& \bP \left( \left| \int_0^{\delta} \ln(\lambda) d \nu_{\left( \Yn  - zI_N \right)} (\lambda) \right| > \epsilon \right) \\
	=& \bP \left( \left| \int_0^{\delta} \ln(\lambda) d \nu_{\left( \Yn  - zI_N \right)} (\lambda) \right| > \epsilon, s_N \left( \Yn  - zI_N \right) < \delta \right) \\
	\le& \bP \left( \dfrac{1}{N} \left| \ln \left( s_N \left( \Yn  - zI_N \right) \right) \right| + \left| \int_0^{\delta} \ln(\lambda) d \nu_{\left( \Zn - zI_N \right)} (\lambda) \right| > \epsilon, s_N \left( \Yn  - zI_N \right) < \delta \right) \\
	\le& \bP \left( \dfrac{1}{N} \left| \ln \left( s_N \left( \Yn  - zI_N \right) \right) \right| + \left| \int_0^{\delta} \ln(\lambda) d \nu_{\left( \Zn - zI_N \right)} (\lambda) \right| > \epsilon, \right. \\
	& \left. n^{-13/11} < s_N \left( \Yn  - zI_N \right) < \delta \right) + Cn^{-1/22} + \bP \left( \|X\| > C_0 \right) \\
	\le& \bP \left( \dfrac{13 \ln n}{11N} + \left| \int_0^{\delta} \ln(\lambda) d \nu_{\left( \Zn - zI_N \right)} (\lambda) \right| > \epsilon \right) + Cn^{-1/22} + \bP \left( \|X\| > C_0 \right).
\end{align*}
In \cite[(35)]{Bose2019}, for all $z \in \bC \setminus \{0\}$, for all $\epsilon > 0$,
\begin{align*}
	\lim_{\delta \rightarrow 0+} \limsup_{n \in \bN_+} \bP \left( \left| \int_0^{\delta} \ln(\lambda) d \nu_{\left( \Zn - zI_N \right)} (\lambda) \right| > \epsilon \right) = 0.
\end{align*}
Thus, when choosing $C_0$ large, we have
\begin{align*}
	\lim_{\delta \rightarrow 0+} \limsup_{n \in \bN_+} \bP \left( \left| \int_0^{\delta} \ln(\lambda) d \nu_{\left( \Yn  - zI_N \right)} (\lambda) \right| > \epsilon \right) = 0,
\end{align*}
where we obtain the uniform integrability of the logarithm with respect to the $\nu_{\left( \Yn  - zI_N \right)}$ in probability, for all $z \in \bC \setminus \{0\}$.

Next, we consider the Hermitian matrix
\begin{align*}
	\Sigma_A(z) = \left(
	\begin{matrix}
	0 & \Yn  - zI_N \\
	\left( \Yn  - zI_N \right)^* & 0
	\end{matrix}
	\right).
\end{align*}
and
\begin{align*}
	\Sigma_J(z) = \left(
	\begin{matrix}
	0 & \Zn - zI_N \\
	\left( \Zn - zI_N \right)^* & 0
	\end{matrix}
	\right).
\end{align*}
Then by \cite{Bose2019}, there exists a probability measure $\nu_z$, such that $\nu_{\left( \Zn - zI_N \right)}$ converges weakly to $\nu_z$ almost surely. Note that the set of eigenvalues of $\Sigma_J$ is
\begin{align*}
	\left\{ \lambda_i(\Sigma_J) : i \in [2N] \right\}
	= \left\{ \pm s_i\left( \Zn - zI_N \right): i \in [N] \right\},
\end{align*}
we have
\begin{align*}
	\mu_{\Sigma_J} (x) = \dfrac{\nu_{\left( \Zn - zI_N \right)} (x) + \nu_{\left( \Zn - zI_N \right)}(-x)}{2}.
\end{align*}
Thus, if we denote by $\check{\nu}_z$ the symmetrization of $\nu_z$, which is the probability measure defined by $\check{\nu}_z(E) = \dfrac{\nu_z(E) + \nu_z(-E)}{2}$ for all Borel set $E$, then $\mu_{\Sigma_J(z)}$ converges weakly to $\check{\nu}_z$ almost surely for almost all $z \in \bC$. Moreover, since $A^{(n)} - J^{(n)}$ has rank one, $\Sigma_A(z) - \Sigma_J(z)$ is a rank two matrix. By the Stability of ESD laws with respect to small rank perturbations (\cite[Exercise 2.4.4]{Tao2012}), we can deduce the weakly convergence of $\mu_{\Sigma_A(z)}$ towards $\check{\nu}_z$ almost surely for almost all $z \in \bC$. Hence, we obtain the weakly convergence of $\nu_{\left( \Yn  - zI_N \right)}$ towards $\nu_z$.

Therefore, by Lemma \ref{Lemma-Hermitization}, $\mu_{\Yn }$ converges weakly to some probability measure $\mu'$ in probability for almost all $z \in \bC$, and the limit measure $\mu'$ is satisfies
\begin{align*}
	\cL_{\mu'}(z) = - \int_0^{\infty} \ln(\lambda) d\nu_z(\lambda).
\end{align*}
Since the singular value empirical distributions of $\Yn  - zI_N$ and $\Zn - zI_N$ converge to the same limit $\nu_z$, we have $\cL_{\mu'}(z) = \cL_{\mu^{(\gamma_0)}}(z)$ for almost all $z \in \bC$. Then by Lemma \ref{Lemma-unique of log potential}, we have $\mu' = \mu^{(\gamma_0)}$.
\end{proof}

\begin{theorem} \label{Thm-LSD-general small k}
Let $N,k$ satisfy \eqref{eq-1.0-N/n} such that $k/n = o(\ln^{-1} n)$. Assume that the conditions (C1) and(C2) hold. then $\mu_{\Yn }$ converges weakly to $\mu^{(\gamma_0)}$ in probability.
\end{theorem}

\begin{proof}
We consider the case $k = 2$ first. The proof is similar to the proof of Theorem \ref{Thm-LSD-k=1}, which is sketched below. Recalled that
\begin{align*}
	\Yn 
	= \sum_{i=1}^{n-2} X_{i+2} X_i^*.
\end{align*}
Let
\begin{align*}
	\widetilde{Y}^{(n)} = Y^{(n)} + X_2 X_{n_o}^*
	= \left( X_1, X_3, \ldots, X_{n_o}, X_2, X_4, \ldots, X_{n_e} \right) \left(
	\begin{matrix}
	0_{1 \times (n-1)} & 0 \\
	I_{n-1} & 0_{(n-1) \times 1}
	\end{matrix}
	\right) \left(
	\begin{matrix}
	X_1^* \\
	\vdots \\
	X_{n_o}^* \\
	X_2^* \\
	\vdots \\
	X_{n_e}^*
	\end{matrix}
	\right),
\end{align*}
where $n_o = 2 \lfloor (n-1)/2 \rfloor + 1$ is the largest odd number that does not exceed $n$, and $n_e = 2 \lfloor n/2 \rfloor$ is the largest even number that does not exceed $n$.

Since $\left( Y^{(n)} - zI_N \right) - \left( \widetilde{Y}^{(n)} - zI_N \right) = X_2X_{n_o}^*$ is a rank one matrix, then by Lemma \ref{Lemma-singular value sum} and the argument in the beginning of the proof of Theorem \ref{Thm-LSD-k=1}, we can obtain
\begin{align*}
	\left| \int_0^{\delta} \ln (\lambda) d\nu_{Y^{(n)} - zI_N} (\lambda) \right|
	\le \dfrac{1}{N} \left| \ln \left( s_N \left( Y^{(n)} - zI_N \right) \right) \right| + \left| \int_0^{\delta} \ln (\lambda) d\nu_{\widetilde{Y}^{(n)} - zI_N} (\lambda) \right|.
\end{align*}
Note that $\left( X_1, X_3, \ldots, X_{n_o}, X_2, X_4, \ldots, X_{n_e} \right) \overset{d}{=} X$, the logarithm function is uniform integrable near zero with respect to $\widetilde{Y}^{(n)} - zI_N$. Hence, by Theorem \ref{Thm-least singular value}, one can obtain the uniform integrability of the logarithm function with respect to $Y^{(n)} - zI_N$ in probability for all $z \in \bC \setminus \{0\}$ by using a similar argument to the proof of Theorem \ref{Thm-LSD-k=1}.

Next, we denote
\begin{align*}
	\Sigma_{Y^{(n)}} (z) = \left(
	\begin{matrix}
	0 & Y^{(n)} - zI_N \\
	(Y^{(n)} - zI_N)^* & 0
	\end{matrix}
	\right), \
	\Sigma_{\widetilde{Y}^{(n)}}(z) = \left(
	\begin{matrix}
	0 & \widetilde{Y}^{(n)} -zI_N \\
	\left( \widetilde{Y}^{(n)} - zI_N \right)^* & 0
	\end{matrix}
	\right).
\end{align*}
Then by the proof of Theorem \ref{Thm-LSD-k=1}, $\mu_{\Sigma_{\widetilde{Y}^{(n)}}(z)}$ converges weakly to $\check{\nu}_z$ almost surely for almost all $z \in \bC$. Since $\Sigma_{Y^{(n)}}(z) - \Sigma_{\widetilde{Y}^{(n)}}(z)$ is a rank two matrix, by the Stability of ESD laws with respect to small rank perturbations (\cite[Exercise 2.4.4]{Tao2012}), we can deduce the weakly convergence of $\mu_{\Sigma_{Y^{(n)}}(z)}$ towards $\check{\nu}_z$ almost surely for almost all $z \in \bC$. Then we obtain the weakly convergence of $\nu_{Y^{(n)} - zI_N}$ towards $\nu_z$ almost surely for almost all $z \in \bC$.

Therefore, by Lemma \ref{Lemma-Hermitization}, $\mu_{Y}$ converges weakly to a probability measure in probability for almost all $z \in \bC$, whose logarithmic potential is the same as $\mu^{(\gamma_0)}$. Then the theorem follows from Lemma \ref{Lemma-unique of log potential}.

The general case $k/n = o\left(\ln^{-1} n\right)$ is similar. We can define the matrix $\widetilde{Y}^{(n)}$ through $Y^{(n)}$ by adding a rank $k-1$ matrix. Similar to the proof of Theorem \ref{Thm-LSD-k=1}, we can obtain
\begin{align*}
	\left| \int_0^{\delta} \ln (\lambda) d\nu_{Y^{(n)} - zI_N} (\lambda) \right|
	\le \dfrac{k-1}{N} \left| \ln \left( s_N \left( Y^{(n)} - zI_N \right) \right) \right| + \left| \int_0^{\delta} \ln (\lambda) d\nu_{\widetilde{Y}^{(n)} - zI_N} (\lambda) \right|,
\end{align*}
which leads to the uniform integrability of the logarithm function with respect to $Y^{(n)}-zI_N$. Since the matrix $\Sigma_{Y^{(n)}}(z) - \Sigma_{\widetilde{Y}^{(n)}}(z)$ has rank $2k$, which is $o(n)$, we can deduce the weakly convergence in probability of $\mu_{\Yn }$ towards $\mu^{(\gamma_0)}$ as the case $k=2$.
\end{proof}

\subsection{The case \texorpdfstring{$k \ge n/2$}{k>=n/2}} \label{sec-LSD-general}

\begin{theorem} \label{Thm-LSD-general k}
Assume that the conditions (C1) and(C2) hold. Let $N,k$ satisfy \eqref{eq-1.0-N/n} such that $k \ge n/2$, then there exists a probability measure $\mu^{(\gamma_0,\gamma_1)}$, sucht that $\mu_{\Yn}$ converges weakly to $\mu^{(\gamma_0,\gamma_1)}$ in probability.
\end{theorem}

\begin{proof}
We apply the logarithmic potential technique (Lemma \ref{Lemma-Hermitization}) to obtain the convergence of $\left\{ \mu_{\Yn}: n\in \bN_+ \right\}$. We divide the proof into two steps. In Step 1, we prove the uniform integrability of the logarithm function for the family $\left\{ \nu_{Y^{(n)} - zI_N}: N \in \bN_+ \right\}$ in probability. Then we prove the almost sure convergence of the singular value empirical measure $\left\{ \nu_{Y^{(n)} - zI_N}: N \in \bN_+ \right\}$ in Step 2.

\textbf{Step 1.}
We still denote
\begin{align*}
	\Sigma_{Y^{(n)}} (z) = \left(
	\begin{matrix}
		0 & Y^{(n)} - zI_N \\
		(Y^{(n)} - zI_N)^* & 0
	\end{matrix}
	\right).
\end{align*}
For $\eta \in \bC_+ = \{w \in \bC: \Im w > 0\}$, we denote the resolvent $G(z,\eta) = \left( \Sigma_{Y^{(n)}} (z) - \eta I_{2N} \right)^{-1}$, then by Lemma \ref{Lemma-Inverse of block matrix}, we have
\begin{align} \label{eq-def-G}
	G(z,\eta)
	= \left(
	\begin{matrix}
		-\eta I_N & Y^{(n)} - zI_N \\
		(Y^{(n)} - zI_N)^* & -\eta I_N
	\end{matrix}
	\right)^{-1}
	= \left(
	\begin{matrix}
		G_{11}(z,\eta), & G_{12}(z,\eta) \\
		G_{21}(z,\eta), & G_{22}(z,\eta) \\
	\end{matrix}
	\right),
\end{align}
where
\begin{align}
	G_{11}(z,\eta) &= \eta \left( \left( Y^{(n)} - zI_N \right) \left( Y^{(n)} - zI_N \right)^* - \eta^2 I_N \right)^{-1}, \label{eq-G11} \\
	G_{12}(z,\eta) &= \left( \left( Y^{(n)} - zI_N \right) \left( Y^{(n)} - zI_N \right)^* - \eta^2 I_N \right)^{-1} \left( Y^{(n)} - zI_N \right) \label{eq-G12} \\
	G_{21}(z,\eta) &= \left( Y^{(n)} - zI_N \right)^* \left( \left( Y^{(n)} - zI_N \right) \left( Y^{(n)} - zI_N \right)^* - \eta^2 I_N \right)^{-1} \label{eq-G21} \\
	G_{22}(z,\eta) &= \eta \left( \left( Y^{(n)} - zI_N \right)^* \left( Y^{(n)} - zI_N \right) - \eta^2 I_N \right)^{-1}. \label{eq-G22}
\end{align}
We first establish the following so-called Wegner estimation
\begin{align} \label{eq-Wegner estimate}
	\dfrac{-\iota}{N} \bE \left[ \Tr \left( G(z, \iota t) \right) \right] \le C \left( 1 + t^{-\alpha} n^{-\beta} \right), \forall z \not= 0, \forall t \in (0,1/2),
\end{align}
for some positive constants $C, \alpha, \beta$.

We follow the idea in \cite{Bose2019} to prove \eqref{eq-Wegner estimate}. By a standard concentration argument (see \cite[Proposition 26]{Bose2019}) one can assume that the entries of $X^{(n)}$ are complex Gaussian. Note that for $k,l,i \in [N]$ and $j \in [n]$, we have
\begin{align}
	\dfrac{\partial \left( G_{11} \right)_{kl}}{\partial \overline{X}_{ij}}
	&= - \left( G_{12} X \left( A^{(n)} \right)^* \right)_{kj} \left( G_{11} \right)_{il} - \left( G_{11} XA^{(n)} \right)_{kj} \left( G_{21} \right)_{il} \label{eq-derivative 1} \\
	\dfrac{\partial \left( G_{12} \right)_{kl}}{\partial \overline{X}_{ij}}
	&= - \left( G_{11} XA^{(n)} \right)_{kj} \left( G_{22} \right)_{il} - \left( G_{12} X \left( A^{(n)} \right)^* \right)_{kj} \left( G_{12} \right)_{il} \label{eq-derivative 2}.
\end{align}
By \eqref{eq-derivative 1}, \eqref{eq-derivative 2} and the Integration by Parts formula for Gaussian variables, for $i, j \in [n]$, we have
\begin{align}
	\bE \left[ X_i^* G_{11} X_j \right]
	&= 1_{\{i=j\}} \bE \left[ \dfrac{1}{n} \Tr G_{11} \right]
	- \bE \left[ \left( X^* G_{12} X \left( A^{(n)} \right)^* \right)_{ij} \cdot \dfrac{1}{n} \Tr G_{11} \right] \nonumber \\
	&\quad - \bE \left[ \left( X^* G_{11} XA^{(n)} \right)_{ij} \cdot \dfrac{1}{n} \Tr G_{21} \right], \label{eq-IP-1} \\
	\bE \left[ X_i^* G_{12} X_j \right]
	&= 1_{\{i=j\}} \bE \left[ \dfrac{1}{n} \Tr G_{12} \right]
	- \bE \left[ \left( X^* G_{11} XA^{(n)} \right)_{ij} \cdot \dfrac{1}{n} \Tr G_{22} \right] \nonumber \\
	&\quad- \bE \left[ \left( X^* G_{12} X \left( A^{(n)} \right)^* \right)_{ij} \cdot \dfrac{1}{n} \Tr G_{12} \right]. \label{eq-IP-2}
\end{align}

By Lemma \ref{Lemma-Poincare-Nash} and \eqref{eq-derivative 1}, we have
\begin{align*}
	\Var \left( \dfrac{1}{n} \Tr G_{11} \right)
	&= \dfrac{1}{n^2} \Var \left( \Tr G_{11} \right) \\
	&\le \dfrac{1}{2n^3} \bE \left[ \sum_{i \in [N], j\in [n]} \left| \dfrac{\partial \Tr G_{11}}{\partial \Re X_{ij}} \right|^2 + \sum_{i \in [N], j\in [n]} \left| \dfrac{\partial \Tr G_{11}}{\partial \Im X_{ij}} \right|^2 \right] \\
	&= \dfrac{1}{n^3} \bE \left[ \sum_{i \in [N], j\in [n]} \left( \left| \dfrac{\partial \Tr G_{11}}{\partial \overline{X}_{ij}} \right|^2 + \left| \dfrac{\partial \Tr G_{11}}{\partial X_{ij}} \right|^2 \right) \right] \\
	&= \dfrac{1}{n^3} \bE \left[ \sum_{i \in [N], j\in [n]} \left( \left| \sum_{k \in [N]} \dfrac{\partial \left( G_{11} \right)_{kk}}{\partial \overline{X}_{ij}} \right|^2 + \left| \sum_{k \in [N]} \dfrac{\partial \left( G_{11} \right)_{kk}}{\partial X_{ij}} \right|^2 \right) \right] \\
	&= \dfrac{2}{n^3} \bE \left[ \sum_{i \in [N], j\in [n]} \left| \left( G_{11} G_{12} X \left( A^{(n)} \right)^* \right)_{ij} + \left( G_{21} G_{11} XA^{(n)} \right)_{ij} \right|^2 \right] \\
	&\le \dfrac{4}{n^3} \bE \left[ \left\| G_{11} G_{12} X \left( A^{(n)} \right)^* \right\|_{HS}^2 + \left\| G_{21} G_{11} XA^{(n)} \right\|_{HS}^2 \right] \\
	&\le \dfrac{4}{n^2} \bE \left[ \left\| G_{11} G_{12} X \left( A^{(n)} \right)^* \right\|^2 + \left\| G_{21} G_{11} XA^{(n)} \right\|^2 \right].
\end{align*}
Here, we use the fact that $\|M\|_{HS} \le \sqrt{n} \|M\|$ for $M \in \bC^{n \times n}$. Note that $\|A^{(n)}\| = 1$, $\|X\| \to 1 + \sqrt{\gamma_0}$, $\left\| \left( \left( Y^{(n)} - zI_N \right) \left( Y^{(n)} - zI_N \right)^* - \eta^2 I_N \right)^{-1} \right\| \le \left( \Im \eta \right)^{-2}$, we have
\begin{align*}
	\Var \left( \dfrac{1}{n} \Tr G_{11} \right)
	\le \dfrac{C' |\eta|^2}{n^2 \left( \Im \eta \right)^{8}}.
\end{align*}
Here, $C'$ is a positive constant that depends only on $\gamma_0$ and $z$ and may vary in different places. By a similar argument, one can obtain
\begin{align} \label{eq-variance of tr G}
	\Var \left( \dfrac{1}{n} \Tr G_{ij} \right)
	\le \dfrac{C' (|\eta|^4 + 1)}{n^2 \left( \Im \eta \right)^{8}}, \ i,j = 1,2.
\end{align}

By Lemma \ref{Lemma-Poincare-Nash} and \eqref{eq-derivative 1}, for $k,l \in [N]$, we have
\begin{align*}
	& \Var \left( X_k^* G_{11} X_l \right) \\
	\le& \dfrac{1}{2n} \bE \left[ \sum_{i \in [N], j\in [n]} \left| \dfrac{\partial \left( X_k^* G_{11} X_l \right)}{\partial \Re X_{ij}} \right|^2 + \sum_{i \in [N], j\in [n]} \left| \dfrac{\partial \left( X_k^* G_{11} X_l \right)}{\partial \Im X_{ij}} \right|^2 \right] \\
	=& \dfrac{1}{n} \bE \left[ \sum_{i \in [N], j\in [n]} \left( \left| \dfrac{\partial \left( X_k^* G_{11} X_l \right)}{\partial \overline{X}_{ij}} \right|^2 + \left| \dfrac{\partial \left( X_k^* G_{11} X_l \right)}{\partial X_{ij}} \right|^2 \right) \right] \\
	=& \dfrac{1}{n} \bE \left[ \sum_{j \in [n]} \left| e_j^{(N)} G_{11}X_l \right|^2 + \sum_{i \in [N], j\in [n]} \left( \left| X_k^* \dfrac{\partial G_{11}}{\partial \overline{X}_{ij}} X_l \right|^2 + \left| X_k^* \dfrac{\partial G_{11}}{\partial X_{ij}} X_l \right|^2 \right) + \sum_{i \in [N]} \left| X_k^* G_{11} e_i^{(N)} \right|^2 \right] \\
	=& \dfrac{1}{n} \bE \Bigg[ 2 \sum_{i \in [N], j\in [n]} \left| \left( X_k^* G_{12} X \left( A^{(n)} \right)^* \right)_j \left( G_{11}X_l \right)_i + \left( X_k^* G_{11} XA^{(n)} \right)_j \left( G_{21}X_l \right)_i \right|^2 \\
	& \qquad\qquad + \left\| G_{11} X_l \right\|^2 + \left\| X_k^* G_{11} \right\|^2 \Bigg] \\
	\le& \dfrac{1}{n} \bE \left[ 4 \left\| X_k^* G_{12} X \left( A^{(n)} \right)^* \right\|^2 \left\| G_{11}X_l \right\|^2 + 4 \left\| X_k^* G_{11} XA^{(n)} \right\|^2 \left\| G_{21}X_l \right\|^2 + \left\| G_{11} X_l \right\|^2 + \left\| X_k^* G_{11} \right\|^2 \right] \\
	\le& \dfrac{1}{n} \bE \Big[ 4 \left\| G_{12}X \right\|^2 \left\| X_k \right\|^2 \left\| G_{11} \right\|^2 \left\| X_l \right\|^2 + 4 \left\| G_{11}X \right\|^2 \left\| X_k \right\|^2 \left\| G_{21} \right\|^2 \left\| X_l \right\|^2 + \left\| G_{11} \right\|^2 \left( \left\| X_l \right\|^2 + \left\| X_k \right\|^2 \right) \Big] \\
	\le& \dfrac{C' |\eta|^2}{n \left( \Im \eta \right)^4} \left( 1 + \dfrac{1}{\left( \Im \eta \right)^4} \right).
\end{align*}
By a similar argument, one can obtain
\begin{align} \label{eq-variance X^*GX}
	\Var \left( X_k^* G_{ij} X_l \right)
	\le \dfrac{C' \left( 1 + |\eta|^4 \right) \left( 1 + \left( \Im \eta \right)^4 \right)}{n \left( \Im \eta \right)^8}, \
	i,j = 1,2.
\end{align}
We denote
\begin{align} \label{eq-def-Psi}
	\Psi(\eta) = \dfrac{\left( 1 + |\eta|^4 \right) \left( 1 + \left( \Im \eta \right)^2 \right)}{\left( \Im \eta \right)^8},
\end{align}
then $\Var \left( \frac{1}{n} \Tr G_{ij} \right) \Var \left( X_k^* G_{i'j'} X_l \right) \le C' \Psi(\eta)^2/n^3$. Hence, substitute \eqref{eq-variance of tr G} and \eqref{eq-variance X^*GX} to \eqref{eq-IP-1} and \eqref{eq-IP-2}, we have
\begin{align}
	\bE \left[ X_i^* G_{11} X_j \right]
	&= 1_{\{i=j\}} \bE \left[ \dfrac{1}{n} \Tr G_{11} \right]
	- \bE \left[ \left( X^* G_{12} X \left(A^{(n)}\right)^* \right)_{ij} \right] \bE \left[ \dfrac{1}{n} \Tr G_{11} \right] \nonumber \\
	&\quad - \bE \left[ \left( X^* G_{11} XA^{(n)} \right)_{ij} \right] \bE \left[ \dfrac{1}{n} \Tr G_{21} \right]
	+ O \left( \dfrac{\Psi(\eta)}{n^{3/2}} \right), \label{eq-entries of A11} \\
	\bE \left[ X_i^* G_{12} X_j \right]
	&= 1_{\{i=j\}} \bE \left[ \dfrac{1}{n} \Tr G_{12} \right]
	- \bE \left[ \left( X^* G_{11} XA^{(n)} \right)_{ij} \right] \bE\left[ \dfrac{1}{n} \Tr G_{22} \right] \nonumber \\
	&\quad - \bE \left[ \left( X^* G_{12} X \left(A^{(n)}\right)^* \right)_{ij} \right] \bE \left[ \dfrac{1}{n} \Tr G_{12} \right]
	+ O \left( \dfrac{\Psi(\eta)}{n^{3/2}} \right). \label{eq-entries of A12}
\end{align}
Here, we use the notation $O(\Psi(\eta)/n^{3/2})$ to represent a number whose absolute value is bounded by $C \Psi(\eta)/n^{3/2}$ for some large constant $C$ that depends on $z$, $\gamma_0$, $\gamma_1$. The constant may vary in different place. We denote $A_{ij} = \bE \left[ X^* G_{ij} X \right]$ and $g_{ij} = \bE \left[ \frac{1}{n} \Tr G_{ij} \right]$ for $i,j = 1,2$. Noting that $g_{11} = g_{22}$, we can write \eqref{eq-entries of A11} and \eqref{eq-entries of A12} as
\begin{align}
	A_{11} \left( I_n + g_{21} A^{(n)} \right) + g_{11} A_{12} \left(A^{(n)}\right)^*
	&= g_{11} I_n + O \left( \dfrac{\Psi(\eta)}{n^{3/2}} \right) E_n, \label{eq-A-1} \\
	g_{11} A_{11} A^{(n)} + A_{12} \left( I_n + g_{12} \left(A^{(n)}\right)^* \right)
	&= g_{12} I_n + O \left( \dfrac{\Psi(\eta)}{n^{3/2}} \right) E_n. \label{eq-A-2}
\end{align}
Here, $E_n$ is a $n \times n$ matrix with all the entries equal to $1$. $O (\Psi(\eta)/n^{3/2}) E_n$ is a matrix whose entries are bounded by $O(\Psi(\eta)/n^{3/2})$. By \eqref{eq-A-1} and \eqref{eq-A-2}, we have
\begin{align}
	A_{11} \left( \left( I_n + g_{21} A^{(n)} \right) \left( I_n + g_{12} \left(A^{(n)}\right)^* \right) - g_{11}^2 A^{(n)} \left(A^{(n)}\right)^* \right)
	&= g_{11} I_n + O \left( \dfrac{\Psi'(\eta)}{n^{3/2}} \right) E_n, \label{eq-equation-A11} \\
	A_{12} \left( \left( I_n + g_{12} \left(A^{(n)}\right)^* \right) \left( I_n + g_{21} A^{(n)} \right) - g_{11}^2 \left(A^{(n)}\right)^* A^{(n)} \right)
	&= g_{12} I_n + \left( g_{12}g_{21} - g_{11}^2 \right) A^{(n)} \nonumber \\
	&\quad + O \left( \dfrac{\Psi'(\eta)}{n^{3/2}} \right) E_n, \label{eq-equation-A12}
\end{align}
where
\begin{align*}
	\Psi'(\eta) = \Psi(\eta) \left( 1 + \dfrac{1}{\left( \Im \eta \right)^2} + \dfrac{|\eta|}{\left( \Im \eta \right)^2} \right).
\end{align*}
To compute $A_{11}$ from \eqref{eq-equation-A11}, we can write
\begin{align*}
	&\quad \left( I_n + g_{21} A^{(n)} \right) \left( I_n + g_{12} \left(A^{(n)}\right)^* \right) - g_{11}^2 A^{(n)} \left(A^{(n)}\right)^* \\
	&= \left(
	\begin{matrix}
		I_k & \left( \begin{matrix} g_{12}I_{n-k} \\ 0_{(2k-n) \times (n-k)} \end{matrix} \right) \\
		\left( g_{21} I_{n-k} \quad 0_{(n-k) \times (2k-n)} \right) & (1 + g_{12}g_{21} - g_{11}^2) I_{n-k}
	\end{matrix}
	\right).
\end{align*}
By Lemma \ref{Lemma-Inverse of block matrix}, if $g_{11}^2 \not= 1$, we have
\begin{align*}
	&\quad \left( \left( I_n + g_{21} A^{(n)} \right) \left( I_n + g_{12} \left(A^{(n)}\right)^* \right) - g_{11}^2 A^{(n)} \left(A^{(n)}\right)^* \right)^{-1} \\
	&= \left(
	\begin{matrix}
		\left( \begin{matrix} \dfrac{1-g_{11}^2+g_{12}g_{21}}{1-g_{11}^2} I_{n-k} & 0_{(n-k) \times (2k-n)} \\ 0_{(2k-n) \times (n-k)} & I_{2k-n} \end{matrix} \right)
		& \left( \begin{matrix} \dfrac{-g_{12}}{1-g_{11}^2} I_{n-k} \\ 0_{(2k-n) \times (n-k)} \end{matrix} \right) \\
		\left( \dfrac{-g_{21}}{1-g_{11}^2} I_{n-k} \quad 0_{(n-k) \times (2k-n)} \right) & \dfrac{1}{1-g_{11}^2} I_{n-k}
	\end{matrix}
	\right).
\end{align*}
Hence, substitute it to \eqref{eq-equation-A11}, we have
\begin{align} \label{eq-A11}
	A_{11} = \left( g_{11} I_n + O \left( \dfrac{\Psi'(\eta)}{n^{3/2}} \right) E_n \right)
	\left(
	\begin{matrix}
		\left( \begin{matrix} \dfrac{1-g_{11}^2+g_{12}g_{21}}{1-g_{11}^2} I_{n-k} & 0_{(n-k) \times (2k-n)} \\ 0_{(2k-n) \times (n-k)} & I_{2k-n} \end{matrix} \right)
		& \left( \begin{matrix} \dfrac{-g_{12}}{1-g_{11}^2} I_{n-k} \\ 0_{(2k-n) \times (n-k)} \end{matrix} \right) \\
		\left( \dfrac{-g_{21}}{1-g_{11}^2} I_{n-k} \quad 0_{(n-k) \times (2k-n)} \right) & \dfrac{1}{1-g_{11}^2} I_{n-k}
	\end{matrix}
	\right).
\end{align}
Similarly, we have
\begin{align} \label{eq-A12}
	A_{12} &= \left( g_{12} I_n + \left( g_{12}g_{21} - g_{11}^2 \right) A^{(n)} + O \left( \dfrac{\Psi'(\eta)}{n^{3/2}} \right) E_n \right) \nonumber \\
	&\quad \left(
	\begin{matrix}
		\dfrac{1}{1-g_{11}^2} I_{n-k}
		& \left( 0_{(n-k) \times (2k-n)} \quad \dfrac{-g_{12}}{1 - g_{11}^2} I_{n-k} \right) \\
		\left( \begin{matrix} 0_{(2k-n) \times (n-k)} \\ \dfrac{-g_{21}}{1-g_{11}^2} I_{n-k} \end{matrix} \right)
		& \left( \begin{matrix} I_{2k-n} & 0_{(2k-n) \times (n-k)} \\ 0_{(n-k) \times (2k-n)} & \dfrac{1 - g_{11}^2 + g_{12}g_{21}}{1 - g_{11}^2} I_{n-k} \end{matrix} \right)
	\end{matrix}
	\right).
\end{align}

In addition, by \eqref{eq-def-G}, we have the following identities.
\begin{align*}
	- \eta G_{11} + G_{12} \left(Y^{(n)}\right)^* - \bar{z} G_{12} &= I_N, \\
	G_{11} Y^{(n)} - zG_{11} - \eta G_{12} &= 0.
\end{align*}
By taking the expectation of trace, we have
\begin{align}
	- \eta g_{11} - \bar{z} g_{12} + \dfrac{1}{n} \Tr \left( A_{12} \left(A^{(n)}\right)^* \right)
	&= \dfrac{N}{n} \label{eq-inverse indentity-1} \\
	-z g_{11} - \eta g_{12} + \dfrac{1}{n} \Tr \left( A_{11} A^{(n)} \right)
	&= 0. \label{eq-inverse indentity-2}
\end{align}

To prove \eqref{eq-Wegner estimate}, we choose $\eta = \iota t$ where $t \in (0,1/2)$. By \eqref{eq-G11}, \eqref{eq-G12}, \eqref{eq-G21} and \eqref{eq-G22}, we can see that $g_{11} (z,\iota t) = g_{22} (z,\iota t)$ is pure imaginary so we can write $g_{11} (z,\iota t) = \iota s(z,t)$ with $s(z,t) > 0$. Furthermore, we have $g_{12} (z,\iota t) = \overline{g_{21} (z,\iota t)}$ and $|g_{12}| \le Ct^{-2}$. By some computation, one can easily see that $\Psi(\iota t) \le 4 t^{-8}$ and $\Psi'(\iota t) \le 12 t^{-10}$. Hence, we can simplify \eqref{eq-A11} as
\begin{align*}
	A_{11} = \iota s \left(
	\begin{matrix}
		\left( \begin{matrix} \dfrac{1+s^2+|g_{12}|^2}{1+s^2} I_{n-k} & 0_{(n-k) \times (2k-n)} \\ 0_{(2k-n) \times (n-k)} & I_{2k-n} \end{matrix} \right)
		& \left( \begin{matrix} \dfrac{-g_{12}}{1+s^2} I_{n-k} \\ 0_{(2k-n) \times (n-k)} \end{matrix} \right) \\
		\left( \dfrac{-g_{21}}{1+s^2} I_{n-k} \quad 0_{(n-k) \times (2k-n)} \right) & \dfrac{1}{1+s^2} I_{n-k}
	\end{matrix}
	\right) + O \left( \dfrac{1}{n^{3/2} t^{14}} \right) E_n.
\end{align*}
Thus,
\begin{align*}
	\dfrac{1}{n} \Tr \left( A_{11} A^{(n)} \right)
	= \dfrac{n-k}{n} \left( \dfrac{-\iota s g_{12}}{1+s^2} \right) + O \left( \dfrac{1}{n^{3/2} t^{14}} \right).
\end{align*}
Together with \eqref{eq-inverse indentity-2}, we have
\begin{align} \label{eq-g12}
	g_{12} = \left( t + \dfrac{n-k}{n} \cdot \dfrac{s}{1+s^2} \right)^{-1} \left( -zs + O \left( \dfrac{1}{n^{3/2} t^{14}} \right) \right).
\end{align}
Similarly, we can simplify \eqref{eq-A12} as
\begin{align*}
	A_{12} = \left(
	\begin{matrix}
		\dfrac{g_{12}}{1+s^2} I_{n-k}
		& \left( 0_{(n-k) \times (2k-n)} \quad \dfrac{-g_{12}^2}{1+s^2} I_{n-k} \right) \\
		\left( \begin{matrix} 0_{(2k-n) \times (n-k)} \\ \dfrac{s^2}{1+s^2} I_{n-k} \end{matrix} \right)
		& \left( \begin{matrix} g_{12} I_{2k-n} & 0_{(2k-n) \times (n-k)} \\ 0_{(n-k) \times (2k-n)} & \dfrac{g_{12}}{1+s^2} I_{n-k} \end{matrix} \right) 
	\end{matrix}
	\right) + O \left( \dfrac{1}{n^{3/2} t^{14}} \right) E_n.
\end{align*}
Hence,
\begin{align*}
	\dfrac{1}{n} \Tr \left( A_{12} \left(A^{(n)}\right)^* \right)
	= \dfrac{n-k}{n} \dfrac{s^2}{1+s^2} + O \left( \dfrac{1}{n^{3/2} t^{14}} \right).
\end{align*}
Together with \eqref{eq-g12} and \eqref{eq-inverse indentity-1}, we have
\begin{align*}
	\left( t + \dfrac{n-k}{n} \cdot \dfrac{s}{1+s^2} \right) \left( ts + \dfrac{n-k}{n} \dfrac{s^2}{1+s^2} - \dfrac{N}{n} \right) + s |z|^2 = O \left( \dfrac{1}{n^{3/2} t^{14}} \right).
\end{align*}
Thus,
\begin{align*}
	s &= \dfrac{\dfrac{N}{n} \left( t + \dfrac{n-k}{n} \cdot \dfrac{s}{1+s^2} \right) - \dfrac{n-k}{n} \cdot \dfrac{2ts^2}{1+s^2} - \left(\dfrac{n-k}{n}\right)^2 \cdot \dfrac{s^3}{\left(1+s^2\right)^2} + O \left( \dfrac{1}{n^{3/2} t^{14}} \right)} {|z|^2 + t^2} \\
	&\le \dfrac{\dfrac{N}{n} \left( t + \dfrac{n-k}{n} \cdot \dfrac{s}{1+s^2} \right) + O \left( \dfrac{1}{n^{3/2} t^{14}} \right)}{|z|^2 + t^2} \\
	&\le \dfrac{N}{n|z|^2} + O \left( \dfrac{1}{n^{3/2} t^{16}} \right),
\end{align*}
which establishes \eqref{eq-Wegner estimate} for $\alpha = 16$ and $\beta = 3/2$.

Therefore, by \cite[Lemma 15]{Guionnet2011} and \eqref{eq-Wegner estimate}, we have
\begin{align} \label{eq-Wegner estimation on large singular value}
	\bE \mu_{\Sigma_{Y^{(n)}}} ((-t,t))
	\le C t \left( 1 + t^{-\alpha} n^{-\beta} \right).
\end{align}
Note that $2\mu_{\Sigma_{Y^{(n)}}} (\cdot) = \nu_{Y^{(n)} - zI_N} (\cdot) + \nu_{Y^{(n)} - zI_N} (-\cdot)$, by a standard argument, we can deduce the following uniform integrability of the logarithm function
\begin{align*}
	\lim_{K \to +\infty} \sup_{n \in \bN_+} \bP \left( \left| \int_{|\ln \lambda| > K} |\ln \lambda| \nu_{Y^{(n)} - zI_N}(d\lambda) \right| > \epsilon \right) = 0, \forall \epsilon > 0.
\end{align*}

\textbf{Step 2.}
By a standard concentration argument (see \cite[Proposition 26]{Bose2019}) one can assume that the entries of $X^{(n)}$ are complex Gaussian. For test function $f \in C_b(\bR)$, for $z \in \bC$, we have
\begin{align} \label{eq-4.1-Singular SD to ESD}
	\int f(x) d\nu_{Y^{(n)} - zI_N}(x)
	&= \int f(\sqrt{x}) d\mu_{\left( Y^{(n)} - zI_N \right) \left( Y^{(n)} - zI_N \right)^*}(x) \nonumber \\
	&= \int f(\sqrt{x}) d\mu_{Y^{(n)} \left(Y^{(n)}\right)^* - z \left(Y^{(n)}\right)^* - \overline{z} Y^{(n)} + |z|^2 I_N}(x) \nonumber \\
	&= \int f(\sqrt{x + |z|^2}) d\mu_{Y^{(n)} \left(Y^{(n)}\right)^* - z \left(Y^{(n)}\right)^* - \overline{z} Y^{(n)}}(x).
\end{align}
By \cite[Theorem 1.3.22]{Horn2013}, we have
\begin{align} \label{eq-4.2-ESD decomposition}
	&\quad \mu_{Y^{(n)} \left(Y^{(n)}\right)^* - z \left(Y^{(n)}\right)^* - \overline{z} Y^{(n)}} \nonumber \\
	&= \mu_{X^{(n)} A^{(n)} \left(X^{(n)}\right)^* X^{(n)} \left(A^{(n)}\right)^* \left(X^{(n)}\right)^* - z X^{(n)} \left(A^{(n)}\right)^* \left(X^{(n)}\right)^* - \overline{z} X^{(n)} A^{(n)} \left(X^{(n)}\right)^*} \nonumber \\
	&= \dfrac{n}{N} \mu_{A^{(n)} \left(X^{(n)}\right)^* X^{(n)} \left(A^{(n)}\right)^* \left(X^{(n)}\right)^* X^{(n)} - z \left(A^{(n)}\right)^* \left(X^{(n)}\right)^* X^{(n)} - \overline{z} A^{(n)} \left(X^{(n)}\right)^* X^{(n)}}
	+ \dfrac{N-n}{N} \delta_0.
\end{align}

Denote $\gamma_0^+ = [\gamma_0] + 1$ be the smallest integer that is strictly greater than $\gamma_0$. Without loss of generality, we can assume that $N < \gamma_0^+ n$. Let $\widehat{X}^{(n)}$ be a $(\gamma_0^+ n) \times n$ matrices whose entries are complex Gaussian with mean zero and variance $1/n$, such that $\widehat{X}^{(n)}_{[N],[n]} = X^{(n)}$. Then Wishart matrix $\left(X^{(n)}\right)^* X^{(n)}$ can be written as
\begin{align*}
	\left(X^{(n)}\right)^* X^{(n)}
	= \left(\widehat{X}^{(n)}\right)^* \left(
	\begin{matrix}
		I_N & 0_{N \times (\gamma_0^+ n-N)} \\
		0_{(\gamma_0^+ n-N) \times N} & 0_{(\gamma_0^+ n-N) \times (\gamma_0^+ n-N)}
	\end{matrix}
	\right) \widehat{X}^{(n)}.
\end{align*}
Under the condition \eqref{eq-1.0-N/n}, for all polynomial $P$ in $2$ non-commutative indeterminates, one has
\begin{align*}
	\dfrac{1}{n} \Tr \left[ P \left( A^{(n)}, \left(A^{(n)}\right)^* \right) \right] \to \tau \left[ P \left( \bfa,\bfa^* \right) \right]
\end{align*}
for some non-commutative element $\bfa$ in a $C^*$-probability space $\left( \cA, \cdot^*, \tau, \|\cdot\| \right)$ with a faithful trace $\tau$. Besides, the Hermitian matrix
\begin{align*}
	\left(
	\begin{matrix}
		I_N & 0_{N \times (\gamma_0^+ n-N)} \\
		0_{(\gamma_0^+ n-N) \times N} & 0_{(\gamma_0^+ n-N) \times (\gamma_0^+ n-N)}
	\end{matrix}
	\right)
\end{align*}
converges to the law
\begin{align*}
	\dfrac{\gamma_0}{\gamma_0^+} \delta_1 + \dfrac{\gamma_0^+ - \gamma_0}{\gamma_0^+} \delta_0
\end{align*}
in the $C^*$-probability space of random matrices. Therefore, by \cite[Corollary 2.2, Theorem 1.6]{Male2012}, there exists a non-commutative random variable $\bfx$ satisfying that $\bfx$ and $\bfa$ are free, such that for any polynomial $P$ in $3$ non-commutative indeterminates,
\begin{align}
	\dfrac{1}{n} \Tr \left[ P \left( \left(X^{(n)}\right)^* X^{(n)}, A^{(n)}, \left(A^{(n)}\right)^* \right) \right]
	\to \tau \left[ P \left( \bfx, \bfa, \bfa^* \right) \right], \
	n \to \infty,
\end{align}
almost surely. Consequently, the eigenvalue empirical measure of $P \left( \left(X^{(n)}\right)^* X^{(n)}, A^{(n)}, \left(A^{(n)}\right)^* \right)$ converges almost surely. In particular, for $z \in \bC$, by choosing the polynomial $P(x,v,w) = vxwx - z wx - \overline{z} vx$, we obtain the almost sure convergence of
\begin{align*}
	\mu_{A^{(n)} \left(X^{(n)}\right)^* X^{(n)} \left(A^{(n)}\right)^* \left(X^{(n)}\right)^* X^{(n)} - z \left(A^{(n)}\right)^* \left(X^{(n)}\right)^* X^{(n)} - \overline{z} A^{(n)} \left(X^{(n)}\right)^* X^{(n)}}.
\end{align*}
Together with \eqref{eq-1.0-N/n}, \eqref{eq-4.2-ESD decomposition} and \eqref{eq-4.1-Singular SD to ESD}, one can easily obtain the almost sure convergence of $\left\{ \nu_{Y^{(n)} - zI_N}: N \in \bN_+ \right\}$.

The proof is concluded by Step 1, Step 2 and Lemma \ref{Lemma-Hermitization}.
\end{proof}

\appendix

\section{Matrices}

The following linear algebraic lemmas could be found in \cite{Bai2010} and \cite{Tao2012}.

\begin{lemma} \label{Lemma-Inverse of block matrix}
For $A \in \bC^{p \times p}$, $D \in \bC^{q \times q}$, $B \in \bC^{p \times q}$, $C \in \bC^{q \times p}$, if $D$ and $A - B D^{-1} C$ are invertible, then
\begin{align*}
	\left(
	\begin{matrix}
		A &B \\
		C &D \\
	\end{matrix}
	\right)^{-1} = \left(
	\begin{matrix}
	(A - B D^{-1} C)^{-1} & - (A - B D^{-1} C)^{-1} B D^{-1} \\
	-D^{-1} C (A - B D^{-1} C)^{-1} & D^{-1} + D^{-1} C (A - B D^{-1} C)^{-1} B D^{-1} \\
	\end{matrix}
	\right).
\end{align*}
\end{lemma}

\begin{lemma} \label{Lemma-singular value sum}
For $A, B \in \bC^{n \times p}$, we have
\begin{align*}
	s_{i+j-1} (A+B) \le s_i(A) + s_j(B), \ \forall i \in \bN_+.
\end{align*}
\end{lemma}

\begin{lemma}{\cite[Exercise 1.3.22]{Bai2010}}
\label{Lemma-singular value interacting}
For $A \in \bC^{m \times n}$, let $C \in \bC^{p \times q}$ be a submatrix of $A$, then singular values of $A$ and $C$ satisfies
\begin{align*}
	s_i(C) \le s_i(A), \ \forall i \in \bN_+.
\end{align*}
\end{lemma}

\section{Concentration inequalities}

The following Hoeffding’s inequality could be found in \cite{Tao2012}.

\begin{lemma} {\cite[Exercise 2.1.4]{Tao2012}} \label{Lemma-Hoeffding concentration}
Let $X1, \ldots, X_n$ be independent real random variables, with $X_i$ taking values in an interval $[0, 1]$, and let $S_n = X_1 + \cdots + X_n$. Then
\begin{align*}
	\bP \left( |S_n - \bE[S_n] | \ge \sqrt{n} \lambda \right) \le C \exp \left( -c \lambda^2 \right),
\end{align*}
for some absolute constants $c,C>0$.
\end{lemma}

The following lemma could be found in \cite{Tao2010}, see also \cite[Lemma 4.13]{Bordenave2012}.

\begin{lemma} {\cite[Lemma 4]{Bose2019}} \label{Lemma-concentration to space}
Let $X$ be a $N \times n$ random matrix satisfying condition (C1). Let $u \in \bS^{n-1}$ be a deterministic vector and $W$ be a deterministic $d$-dimensional vector subspace of $\bC^N$, where $d$ does not depend on $N$ and $n$. Then for large $n, N$,
\begin{align*}
	\bP \left( \dist \left( Xu, W \right) \le c \right) \le \exp (-cN).
\end{align*}
\end{lemma}

\section{Other lemmas}
The following lemma could be found in \cite{Vershynin2014}.
\begin{lemma}{\cite[Lemma 8.3]{Vershynin2014}} \label{Lemma-probability of weight rv}
Let $Z_1, \ldots, Z_n$ be arbitrary non‐negative random variables (not necessarily independent), and $p_1, \ldots, p_n$ be non‐negative numbers such that their sum equals to $1$. Then for every $t \in \bR$,
\begin{align*}
	\bP \left( \sum_{i=1}^n p_i Z_i \le t \right)
	< 2 \sum_{i=1}^n p_i \bP \left( Z_i \le 2t \right).
\end{align*}
\end{lemma}

The following lemma is known as Poincar\'{e}-Nash inequality and could be found in \cite{Pastur2005}.

\begin{lemma}{\cite[Proposition 2.4]{Pastur2005}} \label{Lemma-Poincare-Nash}
Let $\xi = (\xi_1, \ldots, \xi_q)^\intercal$ be a real centered Gaussian random vector with covariance matrix $\Sigma$. Let $\Phi_1, \Phi_2: \bR^q \to \bC$ be two functions with bounded partial derivatives, then
\begin{align*}
	\Cov \left( \Phi_1, \Phi_2 \right)
	= \bE \left[ \Phi_1 \Phi_2 \right] - \bE \left[ \Phi_1 \right] \bE \left[ \Phi_2 \right]
	\le \sqrt{\bE \left[ \left( \nabla \Phi_1 (\xi) \right)^\intercal \Sigma \left( \nabla \Phi_1 (\xi) \right) \right]}
	\sqrt{\bE \left[ \left( \nabla \Phi_2 (\xi) \right)^\intercal \Sigma \left( \nabla \Phi_2 (\xi) \right) \right]}.
\end{align*}
\end{lemma}

\bigskip

\noindent \textbf{Acknowledgments} \quad
We would like to thank Walid Hachem for pointing to us the fact that the least singular value estimate in Section \ref{sec-least singular value} does not require the existence of density for the matrix entries using an approximation argument (see the beginning of the proof of Theorem \ref{Thm-least singular value}).

\bigskip

\bibliographystyle{plainnat}
\bibliography{NonHermitian}

\end{document}